\documentclass[a4paper,11pt]{article}
\usepackage[margin=1in]{geometry}
\usepackage{graphicx,url,subfig}
\RequirePackage[OT1]{fontenc} 
\RequirePackage{amsthm,amsmath,amssymb,amscd}
\RequirePackage[numbers]{natbib}
\RequirePackage[colorlinks,citecolor=blue,urlcolor=blue]{hyperref}
\usepackage{enumerate}
\usepackage{datetime}
\usepackage{etex}
\usepackage{rotating}

\usepackage{bbm}

\makeatother
\numberwithin{equation}{section} 
\allowdisplaybreaks
 
\theoremstyle{plain}
\newtheorem{theorem}{Theorem}[section]
\newtheorem{corollary}[theorem]{Corollary}
\newtheorem{lemma}[theorem]{Lemma}
\newtheorem{proposition}[theorem]{Proposition}

\theoremstyle{definition}
\newtheorem{definition}[theorem]{Definition}

\newtheorem{assumption}[theorem]{Assumption}

\newtheorem{remark}[theorem]{Remark}

\newtheorem{example}[theorem]{Example}

\newcommand{\E}{\mathbb{E}}

\newcommand{\ud}{\ensuremath{\mathrm{d}}}

\newcommand{\sgn}{\text{sgn}}

\newcommand{\Norm}[1]{\left|\left|  #1   \right|\right|}

\newcommand{\InPrd}[1]{\left\langle #1 \right\rangle}

\newcommand{\sprt}{\operatorname{supp}}

\newcommand{\calB}{\mathcal{B}}

\newcommand{\calF}{\mathcal{F}}

\newcommand{\calN}{\mathcal{N}}

\newcommand{\bbF}{\mathbb{F}}
\newcommand{\bbN}{\mathbb{N}}
\newcommand{\bbZ}{\mathbb{Z}}

\newcommand*{\one}{{{\rm 1\mkern-1.5mu}\!{\rm I}}}

\newcommand{\bbP}{\mathbb{P}}

\newcommand{\R}{\mathbb{R}}
\newcommand{\Z}{\mathbb{Z}}

\newcommand{\RR}{\mathbb{R}}

\newcommand{\sign}{\text{sign}}

\DeclareMathOperator{\Lip}{Lip}
\DeclareMathOperator{\LIP}{Lip}

\DeclareMathOperator{\Vip}{\overline{\varsigma}}

\title{Stochastic comparisons for stochastic heat equation  
}
\author{Le Chen\\Emory University\\
Atlanta, Georgia, USA.\\
\texttt{le.chen@emory.edu}\\
\and
Kunwoo Kim\footnote{Research supported by the NRF (National Research Foundation of Korea) grant 2017R1C1B1005436.}\\Pohang University of Science\\ and Technology, Korea\\\texttt{kunwoo@postech.ac.kr}}
\date{\small \today}
\begin{document}
\maketitle
\vspace{-2em}
\begin{center}
\begin{minipage}[rct]{5 in}
\footnotesize \textbf{Abstract:}
We establish the stochastic comparison principles, including moment comparison principle as a special case, for solutions to the following nonlinear stochastic heat equation on $\mathbb{R}^d$
\[
\left(\frac{\partial }{\partial t} -\frac{1}{2}\Delta \right) u(t,x) =  \rho(u(t,x))
\:\dot{M}(t,x),
\]
where $\dot{M}$ is a spatially homogeneous Gaussian noise that is white in time and colored in space, and $\rho$ is a Lipschitz continuous function that vanishes at zero. These results are obtained for rough initial data and under Dalang's condition, namely,  
$\int_{\mathbb{R}^d}(1+|\xi|^2)^{-1}\hat{f}(\text{d} \xi)<\infty$,
where $\hat{f}$ is the spectral measure of the noise. We establish the comparison principles by comparing  either  the diffusion coefficient $\rho$ or the correlation function of the noise $f$. As corollaries, we obtain Slepian's inequality for SPDEs and SDEs. 

\vspace{2ex}

\vspace{2ex}
\vspace{4ex}
\end{minipage}
\end{center}

% \tableofcontents
\setlength{\parindent}{1.5em}

%%%%%%%%%%%%%%%%%%%%%%%%%%%%%%%%%%%%%%%%%%%%%%
%%%%% MAIN: The chapters of the thesis
%%%%%%%%%%%%%%%%%%%%%%%%%%%%%%%%%%%%%%%%%%%%%%

% \mainmatter
% \graphicspath{{../figs/}}
{\hypersetup{linkcolor=black}
\tableofcontents}

\section{Introduction}\label{S:Intro}
In this paper, we study the {\em stochastic comparison principle} (see Definition \ref{D:StochComp}) including {\it moment comparison principle} for the solutions to the following stochastic heat equation (SHE)
\begin{align}\label{E:SHE}
\begin{cases} 
\displaystyle \left(\frac{\partial }{\partial t} -
\frac{1}{2}\Delta \right) u(t,x) =  \rho(u(t,x))
\:\dot{M}(t,x),&
x\in \R^d,\; t>0, \\
\displaystyle \quad u(0,\cdot) = \mu(\cdot).
\end{cases}
\end{align}
In this equation, $\rho$ is assumed to be a globally Lipschitz continuous function with 
\begin{align}\label{E:rho0=0}
\rho(0)=0.
\end{align}
The linear case, i.e., $\rho(u)=\lambda u$, is called 
the {\em parabolic Anderson model} (PAM) \cite{CarmonaMolchanov94}.
The noise $\dot{M}$ is a Gaussian noise that is white in time and homogeneously colored in space.
Informally,
\[
\E\left[\dot{M}(t,x)\dot{M}(s,y)\right] = \delta_0(t-s)f(x-y),
\]
where $\delta_0$ is the Dirac delta measure with unit mass at zero and $f$ is a nontrivial ``correlation function/measure''
i.e., a nonnegative and nonnegative definite function/measure that is not identically zero \footnote{In the following, the terminology ``{\it correlation function}'' should be understood in the generalized sense, i.e., it refers a function  -- the Radon–Nikodym derivative -- when $f$ is absolutely continuous with respect to the Lebesgue measure; otherwise, it refers to a genuine measure.}. 
The Fourier transform of $f$, which is again a nonnegative and nonnegative definite measure and is usually called  the {\it spectral measure}, is denoted by $\hat{f}$
\[
\hat{f}(\xi)= \calF f (\xi) =\int_{\R^d}\exp\left(- i \: \xi\cdot x \right)f(x)\ud x.
\] 
The SPDE \eqref{E:SHE} is understood in its integral form, i.e., {\it the mild solution}, 
\begin{align}\label{E:mild}
u(t,x) = \int_{\R^d} G(t,x-y)\mu(\ud y) + \int_0^t\int_{\R^d} G(t-s,x-y) \rho(u(s,y))M(\ud s, \ud y),
\end{align}
where $G(t,x)$ is the heat kernel function  
\begin{align}\label{E:HeatKernel}
G(t,x) := (2\pi t)^{-d/2}\exp\left(-|x|^2/(2t)\right),
\end{align}
and the stochastic integral is in the sense of Walsh \cite{Dalang99,Walsh}.

\bigskip
We are interested in the stochastic comparison principles for \eqref{E:SHE}, 
which should not be confused with the {\it sample path comparison principle} \cite{CH18Comparison,CK14Comp,Mueller91, Shiga} where one compare solutions for the same equation but with different and comparable initial conditions. We consider 
under either one of the following two scenarios:
\begin{enumerate}[(S-1)]  
 \item  Let $u_1$ and $u_2$ be two solutions to \eqref{E:SHE} with the same (nonnegative) initial data and the same noise but with different diffusion coefficients, namely, $\rho_1$ and $\rho_2$, respectively.
 Assume that 
 \begin{gather*}
\text{either} \quad \rho_1(x)\ge \rho_2(x)\ge 0\quad
\text{or} \quad
\rho_1(x)\le \rho_2(x)\le 0\quad
\text{for all $x\ge 0$.}
\end{gather*}
\item   Let $u_1$ and $u_2$ be two solutions to \eqref{E:SHE} with the same (nonnegative) initial data and the same diffusion coefficient, but with different correlation functions, namely, $f_1$ and $f_2$, respectively.
Assume that
\[
f_1 \geq f_2\quad
\text{(i.e., $f_1-f_2$ is a nonnegative measure).}
\] 
\end{enumerate}
We plan to work under weakest possible conditions on \eqref{E:SHE}, which include {\it rough initial data} and {\it Dalang's condition} on $f$. Let us  explain these two conditions in more details. 
We first note that by the Jordan decomposition, 
any signed Borel measure $\mu$ can be decomposed as $\mu=\mu_+-\mu_-$ where
$\mu_\pm$ are two non-negative Borel measures with disjoint support. Denote $|\mu|:= \mu_++\mu_-$. 
The rough initial data refers to any signed Borel measure $\mu$ such that
\begin{align}\label{E:J0finite}
\int_{\R^d} e^{-a |x|^2} |\mu|(\ud x)
<+\infty\;, \quad \text{for all $a>0$}\;,
\end{align}
where $|x|=\sqrt{x_1^2+\dots+x_d^2}$ denotes the Euclidean norm.
It is easy to see that the condition \eqref{E:J0finite} is equivalent to the condition
that the solution to the homogeneous equation -- $J_0(t,x)$ defined in \eqref{E:J0} below -- exists for all $t>0$ and $x\in\R^d$.
Existence and uniqueness of a random field solution for rough initial conditions are recently established in \cite{CK15SHE} (see also \cite{CH18Comparison} and \cite{Huang}) under Dalang's condition \cite{Dalang99}, i.e.,
\begin{align}\label{E:Dalang}
\Upsilon(\beta):=(2\pi)^{-d}\int_{\R^d} \frac{\hat{f}(\ud \xi)}{\beta+|\xi|^2}<+\infty \quad \text{for some and hence for all $\beta>0$;}
\end{align}
% see Definition \ref{D:Solution} below for the precise meaning of the solution.
Dalang's condition \eqref{E:Dalang} is the weakest condition for the correlation function $f$ in order to have a random field solution (in the sense of Definition \ref{D:Solution}). 
Throughout this paper, we will assume that $\mu$ is a nonnegative measure.

\bigskip 
Instead of presenting our results in full details, which will be done in Section \ref{S:MainRes}, let us first take a look of several examples. 
Under Dalang's condition and for rough initial data, for either one of the above two scenarios (S-1) or (S-2), we have the following comparison results:
\begin{enumerate}[(E-1)]
 \item ({\it Moment comparison principle}) For $m$ arbitrary space-time points $(t_\ell,x_\ell)\in (0,\infty)\times\R^d$ (not necessarily to be distinct) and $m$ integers $k_\ell\in\bbN$, $\ell=1,\cdots,m$, it holds that 
 \begin{align}\label{E:MomComp}
 \E\left[\prod_{\ell=1}^m u_1^{k_\ell}(t_\ell,x_\ell)\right]
 \ge 
 \E\left[\prod_{\ell=1}^m u_2^{k_\ell}(t_\ell,x_\ell)\right].
\end{align} 
\item For any $(t,x)\in(0,\infty)\times\R^d$, $c> 0$ and any integer $n\ge 1$, it holds that
\begin{align}\label{E:C5}
\E\left(\left[u_1(t,x)-c\right]^{2n}\right)
\ge 
\E\left(\left[u_2(t,x)-c\right]^{2n}\right).
\end{align}
In particular, by choosing $c=J_0(t,x)$ (see \eqref{E:J0} below), \eqref{E:C5} tells us that all {\it central moments} of even orders can be compared. 
When $n=1$, this is a comparison result for the variances. 
\item 
For $m$ arbitrary space-time points $(t_\ell,x_\ell)\in (0,\infty)\times\R^d$ (not necessarily to be distinct) and $m$ integers $k_\ell\in\bbN$, $\ell=1,\cdots,m$, it holds that 
\begin{align}\label{E:CompCoordinate}
\E\left[\prod_{\ell=1}^m g_\ell^{k_\ell}(u_1(t_\ell,x_\ell))\right]
 \ge 
 \E\left[\prod_{\ell=1}^m g_\ell^{k_\ell}(u_2(t_\ell,x_\ell))\right],
\end{align}
where $g_\ell(z)$ can be any of the following functions 
\[
\exp\left(-\lambda_\ell z\right),\quad \frac{1}{(1+z)^{c_\ell}},\quad\text{or}\quad \log\left(\frac{z+a_\ell}{z+b_\ell}\right)\qquad\text{with $\lambda_\ell>0$, $a_\ell>b_\ell>0$ and $c_\ell\ge 1$.}
\]
\item Statement in \eqref{E:CompCoordinate} is true with $g_\ell(z)$ being either of the following two functions:
\[
x^{b_\ell}[\log(c_\ell+ x)]^{a_\ell}
\quad\text{or} \quad x^{d_\ell},\qquad \text{with $a_\ell, b_\ell,d_\ell\ge  1$ and $c_\ell \ge e$.}
\]
\item For any $m\ge 1$, $t_m>\cdots >t_1>0$,  $k_1,\cdots k_m\in\bbN\setminus\{0\}$, $\alpha_1,\cdots,\alpha_m\in [2,\infty)$, and $x_{j}^\ell\in\R^d$ with $\ell=1,\cdots,m$ and $j=1,\cdots,k_\ell$ such that $\left\{x_{k_1}^\ell,\cdots,x_{k_\ell}^\ell\right\}$ are distinct points for each $\ell$, 
\begin{multline}\label{E:C4}
\E\left(
\prod_{\ell=1}^m \left[u_1^2\left(t_\ell,x_{1}^\ell\right)+\cdots+ u_1^2\left(t_\ell,x_{k_\ell}^\ell\right)\right]^{\frac{\alpha_\ell}{2}}
\right)\\
\ge 
\E\left(
\prod_{\ell=1}^m \left[u_2^2\left(t_\ell,x_{1}^\ell \right)+ \cdots+u_2^2\left(t_\ell,x_{k_\ell}^\ell \right)\right]^{\frac{\alpha_\ell}{2}}
\right).
\end{multline}
\end{enumerate} 

Note that \eqref{E:C5} is not a special case of \eqref{E:MomComp} when $n\ge 2$. The oscillatory nature caused by the negative one makes \eqref{E:C5} non-trivial.
One more example, that is slightly different from the above ones, is the following  Slepian's inequality for SPDEs:
\begin{enumerate}
 \item[(E-6)] {\it (Slepian's inequality for SPDEs)} Under the scenario (S-2), if $f_1$ and $f_2$ are {\it equal} to each other near the origin (see the precise meaning in Corollary \ref{C:SPDE:Slepian} below), then for all $a>0$, $t>0$, and $x_1,\cdots, x_N\in\R^d$, 
 \begin{equation}\label{E:SPDE:Slepian2}
\bbP \left\{ \max_{1\leq k \leq N} u_1(t, x_k) \leq a  \right\} \geq \bbP \left\{ \max_{1\leq k \leq N} u_2(t, x_k) \leq a  \right\}. 
\end{equation}
\end{enumerate}

% \bigskip
For the parabolic Anderson model (PAM), it is well known that the moments enjoy the Feynman-Kac representation, based on which one can obtain very sharp estimates for the moments. The literature is vast and we refer the interested readers to Xia Chen's papers \cite{XChen15,XChen19}
and references therein.  One may also check the work by Borodin and  Corwin \cite{BC14Moment} where the $p$-th moment is represented by some multiple contour integrals.  Using the sharp estimates of the moments for PAM, intermittent phenomena (i.e., the solution develops tall peaks on small \emph{islands} of many different scales), have been studied extensively, e.g., see \cite{CarmonaMolchanov94, FK13SHE} for the definition and analysis of intermittency in terms of moments and also  \cite{KKX1, KKX2, Kim} for the study on intermittency based on the macroscopic multi-fractal analysis.   However, whenever $\rho$ is nonlinear or whenever the functionals go beyond the moments functionals, much fewer tools are available. The stochastic comparison results of the above kinds, including moment comparison principle, play a fundamental role in this setting. 
\bigskip

When the noise is additive, i.e., $\rho(u)=\text{constant}$, the moment comparison principle --- Case (E-1) --- under the second scenario (S-2) comes from Isserlis' theorem \cite{Isserlis} since the solution is a Gaussian random field whose distribution is determined by the spatial correlation function $f$. 
On the other hand, to the best of our knowledge, the comparison principle including the moment comparison principle under the second scenario is new for  \eqref{E:SHE} with the condition \eqref{E:rho0=0}.  As for the first scenario, the moment comparisons principle --- Case (E-1) --- has been studied recently. 
In \cite{JKM17}, Joseph, Khoshnevisan and Mueller proved one-time comparison of \eqref{E:MomComp} for 
the one-dimensional case, i.e., $d=1$, 
with space-time white noise $f=\delta_0$, and $t_1=\cdots=t_m$, which was later generalized by Foondun, Joseph and  Li in 
\cite{FJL18} to the multiple-time comparison of the form \eqref{E:MomComp} in the higher dimensional case $d\ge 1$ with the {\it Riesz 
kernel} 
\begin{align}
 f(x)=|x|^{-\beta} \qquad \text{with}\quad  \beta\in (0,2\wedge d).
\end{align}
It is easy to see that the Riesz kernel with the range of $\beta$ specified above satisfies Dalang's condition \eqref{E:Dalang}.
In both \cite{JKM17} and \cite{FJL18}, the initial conditions are assumed to be 
the Lebesgue measure $\mu(\ud x) =\ud x$. 
We will generalize these results to cover rough initial data and all possible correlation functions under Dalang's condition \eqref{E:Dalang}. Moreover, we will cover many other functionals other than moment functionals in Case (E-1).
The approximation results in Sections \ref{S:Aprox} and \ref{SS:Aprox2} below are interesting by themselves, where we use different approximation procedures which produce  strong solutions in this paper  rather than mild solutions as in \cite{JKM17, FJL18}.  We believe  strong solutions are more straightforward and easier to handle when showing  approximations.

\subsection{Statement of the main results}\label{S:MainRes}
In order to state our main results, we first need to introduce some notation.
We first note that under our assumptions, namely, $\rho(0)=0$ and the initial data $\mu$ being nonnegative, the solutions to \eqref{E:SHE} are nonnegative (see \cite{CH18Comparison,CK14Comp} and also Theorem \ref{T:WeakComp} below). Hence, all function spaces in Definition \ref{D:Functions} have their domains in $\R_+^m$ for some $m\ge 1$. 

\begin{definition}\label{D:Functions}
For $m\ge 1$, let $C^{2,v}\left(\R_+^m;\R_+\right)$ be the set of nonnegative  functions on $\R_+^m$ having continuous second order partial derivatives and  all second order partial derivatives are nonnegative.  Let $C_b^{2,v}\left(\R_+^m;\R_+\right)$ be the set of functions in $C^{2,v}\left(\R_+^m;\R_+\right)$ such that all partial derivatives of orders $0$, $1$ and $2$ are bounded.
Let $C_p^{2,v}\left(\R_+^m;\R_+\right)$ be the set of functions in $C^{2,v}\left(\R_+^m;\R_+\right)$ such that the gradient has at most some polynomial growth, namely, $f\in C_p^{2,v}(\R_+^m;\R_+)$, then there exists some constant $C>0$ and $k\in\bbN$ such that
 \begin{align}\label{E:GrowthCp}
 |\bigtriangledown f (z)| \le C(1+|z|^k),\quad\text{for all $z\in\R_+^m$.}
 \end{align}
Let $C_{-}^{2,v}\left(\R_+^m;\R_+\right)$ (resp. $C_{+}^{2,v}\left(\R_+^m;\R_+\right)$) be the set of functions in $C^{2,v}\left(\R_+^m;\R_+\right)$ such that all first derivatives are non-positive (resp. nonnegative) and set 
$C_{\pm}^{2,v}\left(\R_+^m;\R_+\right):=C_{+}^{2,v}\left(\R_+^m;\R_+\right)\cup C_{-}^{2,v}\left(\R_+^m;\R_+\right)$. 
Similarly, one can define 
\begin{align*}
&C_{b,-}^{2,v}\left(\R_+^m;\R_+\right),\quad
C_{b,+}^{2,v}\left(\R_+^m;\R_+\right),\quad
C_{b,\pm}^{2,v}\left(\R_+^m;\R_+\right),\quad \text{and}\\
&C_{p,-}^{2,v}\left(\R_+^m;\R_+\right),\quad 
C_{p,+}^{2,v}\left(\R_+^m;\R_+\right),\quad 
C_{p,\pm}^{2,v}\left(\R_+^m;\R_+\right).
\end{align*}
\end{definition}

\begin{definition}
\label{D:Cones}
Let $K$ be the spatial index set, which could be either $\R^d$ or $\bbZ^d$ or a finite set $\{0,\cdots,d\}$. 
Let $\mathbb{F}^K[C^{2,v}]$ denote the set of finite-dimensional nonnegative functions of twice continuously differentiable functions, namely,
\begin{align}
\begin{aligned}
\bigcup_{m=1}^{|K|}
\mathop{\mathop{\bigcup_{x_\ell\in K:}}_{\ell=1,\cdots, m,}}_{x_i\ne x_j,\: i\ne j}
\bigg\{ f: \R_+^K \mapsto \R_+: 
\exists\:  g\in C^{2,v}(\R_+^m;\R_+)\:\: \text{s.t.}\:\: 
f(z)=g\left(z(x_1),\cdots, z(x_m)\right) \bigg\}
\end{aligned}
\label{E:F}
\end{align}
where $|K|$ is the cardinality of the index set $K$, which is equal to $d$ for $K=\{1,\cdots, d\}$ and $\infty$ when there is countably or uncountable many elements in $K$.
In the same way, one can define 
\begin{align}\label{E:Cones}
\mathbb{F}^K[C_{+}^{2,v}],\quad 
\mathbb{F}^K[C_{-}^{2,v}],\quad
\mathbb{F}^K[C_b^{2,v}],\quad 
\mathbb{F}^K[C_{b,+}^{2,v}],\quad 
\mathbb{F}^K[C_{b,-}^{2,v}],
\quad 
\mathbb{F}^K[C_p^{2,v}],\quad 
\mathbb{F}^K[C_{p,+}^{2,v}].
\end{align}
Let $\mathbb{F}_M$ and $\mathbb{F}_L$ denote the set of {\it moment and Laplace functions}, i.e.,  
\begin{align}\label{E:FM}
\mathbb{F}_M^K:= & \bigcup_{m=1}^{|K|}
\mathop{\mathop{\bigcup_{(k_\ell,x_\ell)\in\bbN\times K:}}_{\ell=1,\cdots, m,}}_{x_i\ne x_j,\: i\ne j}
\left\{ f: \R_+^K \mapsto \R_+: f(z)=z(x_1)^{k_1}\cdots z(x_m)^{k_m}\right\},\\
\mathbb{F}_L^K:= & \bigcup_{m=1}^{|K|}
\mathop{\mathop{\bigcup_{(\lambda_\ell,x_\ell)\in\R_+\times K:}}_{\ell=1,\cdots, m,}}_{x_i\ne x_j,\: i\ne j}
\left\{ f: \R_+^K \mapsto \R_+: f(z)=\exp\left( -\sum_{\ell=1}^m \lambda_\ell \: z(x_\ell)\right) \right\}.
\label{E:FL}
\end{align}
\end{definition}

When there is no ambiguity from the context, we often omit the superscript $K$ for these function spaces.

\begin{remark}
In \cite{CFG96}, $\mathbb{F}[C^{2,v}]$ and any one in \eqref{E:Cones} are {\it function cones} because we will see latter that they are preserved under the certain semigroup operations (and/or multiplication). In contrast, $\mathbb{F}_M$ and $\mathbb{F}_L$ are not cones in that sense.   
It is clear that these sets of functions satisfy the following inclusion relations:
\begin{align}\label{E:Rel-Fs}
\begin{array}{ccccccc}
  &&\bbF[C_+^{2,v}] & \subseteq & \bbF[C_{\pm}^{2,v}] & \subseteq & \bbF[C^{2,v}]\\[0.2em]
  &&\text{\begin{rotate}{90}$\subseteq$\end{rotate}}&& \text{\begin{rotate}{90}$\subseteq$\end{rotate}}
 & &\text{\begin{rotate}{90}$\subseteq$\end{rotate}}\\
 \bbF_M &\subseteq & \mathbb{F}[C_{p,+}^{2,v}]& \subseteq & \bbF[C_{p,\pm}^{2,v}] & \subseteq & \bbF[C_p^{2,v}]\\[0.2em]
  &&&& \text{\begin{rotate}{90}$\subseteq$\end{rotate}}
 & &\text{\begin{rotate}{90}$\subseteq$\end{rotate}}\\
 \bbF_L &\subseteq & \mathbb{F}[C_{b,-}^{2,v}]&\subseteq & \bbF[C_{b,\pm}^{2,v}] & \subseteq & \bbF[C_b^{2,v}]
\end{array}
\end{align}
% By modifying the $f$ in \eqref{E:FL} slightly, for example $f(z)=\exp\left( -\sum_{\ell=1}^m \lambda_\ell \: (z(x_\ell)-a_\ell)^2\right)$ with some $a_\ell>0$, one finds examples that belong to $\bbF[C_b^{2,v}]$ but not to $\bbF[C_{b,\pm}^{2,v}]$.
\end{remark}

\begin{definition}\label{D:StochComp}
Let $\{u_i(t,x); (t,x)\in \R_+\times K\}$, $i=1,2$, be two random fields, where $K$ is the spatial index set as in Definition \ref{D:Cones}.
For some set of functions $\bbF$, such as those defined in Definition \ref{D:Cones}, and for some $n\ge 1$, we say that $u_1$ and $u_2$ satisfy the {\it $n$-time stochastic comparison principle over $\bbF$ with $u_1$ dominating $u_2$} if for any $0<t_1<\cdots<t_n<\infty$, and $F_1,\dots, F_n\in\bbF$, 
it holds that
\begin{align}\label{E!:MomComp}
 \E\left[\prod_{\ell=1}^n F_\ell\left(u_1(t_\ell,\cdot)\right)\right] 
 \ge 
 \E\left[\prod_{\ell=1}^n F_\ell\left(u_2(t_\ell,\cdot)\right)\right].
\end{align}
\end{definition}

Now we are ready to state our main results: 

\begin{theorem}[Comparison with respect to diffusion coefficients]
\label{T:MomComp} 
Suppose that the correlation function $f$ satisfies Dalang's condition \eqref{E:Dalang}. 
Let $\mu$ be a nonnegative measure that satisfies \eqref{E:J0finite}.
Let $u_1(t,x)$ and $u_2(t,x)$ be two solutions of \eqref{E:SHE}, both starting from $\mu$, but with diffusion coefficients $\rho_1$ and $\rho_2$, respectively. If 
\begin{gather*}
\text{either} \quad \rho_1(x)\ge \rho_2(x)\ge 0\quad
\text{or} \quad
\rho_1(x)\le \rho_2(x)\le 0\quad
\text{for all $x\ge 0$,}
\end{gather*}
then for any integer $n\ge 1$, $u_1$ and $u_2$ satisfy the $n$-time (resp. $1$-time) stochastic comparison principle over either $\bbF[C_{p,+}^{2,v}]$ or $\bbF[C^{2,v}_{b,-}]$ (resp. $\bbF[C^{2,v}_p]$) with $u_1$ dominating $u_2$, where the spatial index set is $K=\R^d$. 
\end{theorem}

We now state the stochastic comparison theorem with respect to two comparable correlation functions $f_1$ and $f_2$.

\begin{theorem}[Comparison with respect to correlations of noises]\label{T2:MomComp}
Let $\dot{M}^{(1)}$ and $\dot{M}^{(2)}$ be two noises with correlation functions $f_1$ and $f_2$, respectively,  that satisfy Dalang's condition \eqref{E:Dalang}.
Let $\mu$ be a nonnegative measure that satisfies \eqref{E:J0finite}.
Let $u_1(t,x)$ and $u_2(t,x)$ be two solutions of \eqref{E:SHE}, both starting from $\mu$, with the same diffusion coefficient $\rho$, driven by $\dot{M}^{(1)}$ and $\dot{M}^{(2)}$, respectively. 
If 
\[
f_1 \geq f_2\quad
\text{(i.e., $f_1-f_2$ is a nonnegative measure),}
\] 
then for any integer $n\ge 1$, $u_1$ and $u_2$ satisfy the $n$-time (resp. $1$-time) stochastic comparison principle over either $\bbF[C_{p,+}^{2,v}]$ or $\bbF[C^{2,v}_{b,-}]$ (resp. $\bbF[C^{2,v}_p]$) with $u_1$ dominating $u_2$, where the spatial index set is $K=\R^d$.
\end{theorem}

We would like to point out that for multiple-time comparison results, working on $\mathbb{F}_M$ alone won't be sufficient since $\mathbb{F}_M$ is not a function cone, i.e., it is not preserved under the underlying semigroup and multiplication (see Step 3 of the proof of Theorem \ref{T:FiniteCc} in Section \ref{SS:Finite} below). One needs to go through  the function cone $\mathbb{F}[C^{2,v}_{+}]$ or $\mathbb{F}[C^{2,v}_{-}]$ as in \cite{CFG96}. On the other had, as an application of the 1-time comparison principle, we can obtain Slepian's inequality for SPDEs. Let $C_b^{2}(\R^d;\R_+)$ denote the set of $C^2$ functions with bounded partial derivatives of orders $0$, $1$ and $2$. %Set $\phi(x) = \one_{\{|x|\le 1\}}$ and $\phi_\epsilon(x)=\epsilon^{-d}\phi(x/\epsilon)$.}

\begin{corollary}[Slepian's inequality for SPDEs] \label{C:SPDE:Slepian}
Under the assumptions in Theorem \ref{T2:MomComp} and, in addition, either 

\begin{enumerate}
\item[(i)] for some $\epsilon>0$ such that $f_1\left([-\epsilon, \epsilon]^d\right) = f_2\left([-\epsilon, \epsilon]^d\right)$ or
\item[(ii)] 
both $f_1$ and $f_2$ are  in $C_b^2(\R^d; \R_+)$ such that $f_1(0)=f_2(0)$,
\end{enumerate}
we have that, for any numbers $a_i > 0$, $x_i \in \R^d$ for $i=1, \dots, N$, and $t\geq 0$,  
\begin{equation}\label{E:SPDE:Slepian1}
\bbP\left\{u_1(t, x_1) \leq a_1, \dots, u_1(t, x_N) \leq a_N    \right\}  \geq  \bbP\left\{u_2(t, x_1) \leq a_1, \dots, u_2(t, x_N) \leq a_N    \right\}. 
\end{equation}
In particular, for any $a>\R$, $x_i \in \R^d$ for $i=1, \dots, N$, and $t\geq 0$, the inequality \eqref{E:SPDE:Slepian2} it true. 
\end{corollary}

Here, one example for the case $(i)$ in Corollary \ref{C:SPDE:Slepian} is that $d=1$, $f_1(x)=\delta_0(x)+c(\delta_{-1}(x)+\delta_1(x))$ and $f_2(x)=\delta_0(x)$ where $c\in [0,1/2]$ is a fixed constant. For the case $(ii)$, $f_1(x)=e^{-|x|^2}$ and $f_2(x)=e^{-2|x|^2}$ or $f_1(x)=\frac{1}{1+|x|^2}$ and $f_2(x)= \frac{1}{1+2|x|^2}$.

\bigskip
{\bf\noindent Interacting diffusions. }
The proof of the above comparison theorems \ref{T:MomComp} and \ref{T2:MomComp} rely on similar comparison results for the following linearly interacting diffusions, which are of interest by themselves.
Let $K$ denote a non-empty set with at most countably infinite elements (e.g. $K=\delta \bbZ^d$ with $\delta>0$ fixed or $K=\{1,\cdots, d\}$). Let us consider 
\begin{equation}\label{E:SDE1}
 \left\{
\begin{array}{ll}
\displaystyle  d U(t, i)=\kappa \sum_{j\in K} p_{i,j} \left( U(t, j)- U(t, i)\right) \ud t + \rho (U(t,i)) \ud M_i(t), \quad &i\in K, t>0, 
\\[1em]
\displaystyle      
U(0, i) =u_0(i)\,, & i \in K, \,\\
\end{array}
\right.
\end{equation}
where $\kappa>0$ is a fixed constant and we make the following assumptions over this equation:

\begin{assumption}\label{A:SDE}
Assume that
\begin{enumerate} 
\item[(i)] $p:=\{p_{i,j}; \ i, j \in K\}$ is a probability transition matrix in $K$ such that
\begin{align}\label{E:Lambda}
\Lambda:= \sup_{j\in K}\sum_{i\in K} p_{i,j}<+\infty.
\end{align}
\item[(ii)] $\rho:\R_+\to \R_+$ is a globally Lipschitz function with $\rho(0)=0$.
\item[(iii)] $\{M_i(t); t\ge 0\}_{i\in K}$ is a system of correlated Brownian motions with the following covariance structure:
\begin{equation}\label{E:corrBM}
\E [M_i(t)M_j(s)]=(t\wedge s) \gamma (i-j),
\end{equation}
where $\gamma: K\to \R_+$ is a non-negative, symmetric and non-negative definite function. 
\item[(iv)] $u_0:K\to \R_+$ is a non-negative function in $\ell^2(K)$ such that $u_0(i)>0$ for some $i\in K$.
\end{enumerate}
\end{assumption} 

\begin{remark}
Regarding condition \eqref{E:Lambda}, when the state space $K$ has finite cardinality, it is trivially satisfied. When the underlying random walk is symmetric, i.e., $p_{i,j}=p_{j,i}$, then this condition is satisfied with $\Lambda=1$. 
\end{remark}

We say that  $U=\{U(t, i);\, t\geq 0, i\in K\}$ is a {\em strong solution} to \eqref{E:SDE1} with the initial data $u_0(\cdot)$ if it satisfies that for all $i\in K$ and $t>0$,
\begin{align}\label{E:SDE1-Strong}
U(t, i) =u_0(i)+ \kappa \int_0^t \sum_{j\in K} p_{i,j}(U(s,j)-U(s, i)) \, \ud s + \int_0^t\rho (U(s,i)) \ud M_i(s).
\end{align}  
The existence and uniqueness of a strong solution to \eqref{E:SDE1} when the driving Brownian motions are independent is well-known (see, e.g., \cite{ShigaShimuzu}).
Since we only need the case when the initial data is in $\ell^2(K)$ --- (iv) of Assumption \ref{A:SDE}, 
we won't need the weighted $\ell^2(K)$ space as was used in \cite{ShigaShimuzu}.  
In \cite{FJL18},  $p_{i,j}$ depends only on $j-i$ and it is shown that there is a unique {\it mild} solution to \eqref{E:SDE1} in $L^\infty([0,T]\times K; L^k(\Omega))$ for any $T>0$ and $k\ge 2$. The next theorem, on the other hand, we provide a proof of existence and uniqueness of a strong solution in a slightly better space (see  \eqref{E:DiscreteSpace}) and for more general transition probabilities $p_{i,j}$. As one can see later, a strong solution is easier to handle than a mild solution when showing approximations.

\begin{theorem}\label{T:ExtUniqSDE1}
There exists a unique strong solution $\{U(t, i); t\ge 0, i\in K\}$ to \eqref{E:SDE1} in 
\begin{align}\label{E:DiscreteSpace}
L^\infty\left([0,T];L^k\left(\Omega;\ell^k(K)\right)\right)
\qquad\text{for any $T>0$ and $k\ge 2$.}
\end{align}
In particular, $U(t,\cdot)\in \ell^k(K)$ a.s. for any $t\geq 0$ and $k\ge 2$. Moreover, for any $T>0$ and $k\geq 2$, 
\begin{align}\label{E:lkMom}
\sup_{0\leq t\leq T} \E \left[\sup_{i\in K}\,|U(t,i)|^k\right]\le \sup_{0\leq t\leq T}  \E\left[ \Norm{U(t,\cdot)}_{\ell^k(K)}^k\right]& \leq 3^k\|u_0\|^k_{\ell^k(K)} \exp\left(C k^2 T \right)<\infty, 
\end{align}
where the constant $C>0$ depends only on $\kappa, \Lip_\rho$, $\gamma(0)$ and $\Lambda$.
\end{theorem}

Note that the discrete nature of the spatial variable enables us to bring the supremum over the spatial variable inside the expectation; see \eqref{E:Sup<Lp}. This is in general not true when the spatial variable lives in $\R^d$.
For this interacting diffusions \eqref{E:SDE1}, we have the following two similar stochastic comparison results: 

\begin{theorem}[Comparison with respect to diffusion coefficients]\label{T:MomComSDE1}
Let $U_1$ and $U_2$ be two solutions to \eqref{E:SDE1}, both starting from $u_0$, but with diffusion coefficients $\rho_1$ and $\rho_2$, respectively. 
Then the condition 
\[
\rho_1(x)\geq \rho_2(x)\ge 0
\quad\text{or}\quad \rho_1(x)\le \rho_2(x)\le 0, \quad\text{for all $x\geq 0$}
\]
implies that for any integer $n\ge 1$, $U_1$ and $U_2$ satisfy the $n$-time (resp. $1$-time) stochastic comparison principle over either $\bbF[C_{p,+}^{2,v}]$ or $\bbF[C^{2,v}_{b,-}]$ (resp. $\bbF[C^{2,v}_p]$) with $U_1$ dominating $U_2$.
\end{theorem} 

\begin{theorem}[Comparison with respect to covariances of noises]\label{T:MomComSDE2}
Let $U_1$ and $U_2$ be two solutions to \eqref{E:SDE1}, both starting from $u_0$, with the same diffusion coefficients 
$\rho$, %that satisfies 
%\[
%\rho(0)=0\quad\text{and}\quad \rho(x)\ge 0\quad\text{for all $x\geq 0$},
%\]
but driven by two sets of correlated Brownian motions $\{M^{(1)}_i(t); t\geq 0\}_{i\in K}$ and $\{M^{(2)}_i(t);t\geq 0\}_{i\in K}$, respectively. Let $\gamma_i$ be the covariance function for $M^{(i)}$.  
Then the condition 
\[ \gamma_1( k )\geq \gamma_2(k), \quad \text{for all $k\in K$}\] implies that for any integer $n\ge 1$, $U_1$ and $U_2$ satisfy the $n$-time (resp. $1$-time) stochastic comparison principle over either $\bbF[C_{p,+}^{2,v}]$ or $\bbF[C^{2,v}_{b,-}]$ (resp. $\bbF[C^{2,v}_p]$) with $U_1$ dominating $U_2$.
\end{theorem}

\begin{corollary}[Slepian's inequality for interacting diffusions]\label{C:Slepian}
Under the assumptions in Theorem \ref{T:MomComSDE2} and, in addition, 
\[\gamma_1(0)=\gamma_2(0),\] 
we have  that, for any numbers $a_k\in \R$, $i_k \in K$ for $k=1, \dots, N$, and  $t\geq 0$, 
\begin{equation}\label{E:Slepian1}
\bbP\left\{U_1(t, i_1) \leq a_1, \dots, U_1(t, i_N) \leq a_N    \right\}  \geq  \bbP\left\{U_2(t, i_1) \leq a_1, \dots, U_2(t, i_N) \leq a_N    \right\}. 
\end{equation}
In particular, we have that, for any $a\in \R$, $i_k \in K$ for $k=1, \dots, N$, and  $t\geq 0$, 
\begin{equation}\label{E:Slepian2}
\bbP \left\{ \max_{1\leq k \leq N} U_1(t, i_k) \leq a  \right\} \geq \bbP \left\{ \max_{1\leq k \leq N} U_2(t, i_k) \leq a  \right\}. 
\end{equation}
\end{corollary}

At the very core of the chain of arguments is the following comparison results for the finite dimensional SDE with $C_c^2(\R_+)$ diffusion coefficient. Here $C_c^2(\R_+)=C_c^2(\R_+;\R)$ refers to the functions
defined on $\R_+$, having compact support and continuous second derivative.

\begin{assumption}\label{A:FiniteCc}
 In the SDE \eqref{E:SDE1}, we assume that (i) $\rho\in C_c^2(\R_+)$ and $\rho(u_0(i))\ne 0$ for some $i\in K$; and (ii) the cardinality of the index set $K$ is finite.
\end{assumption}

\begin{theorem}\label{T:FiniteCc}
Under Assumption \ref{A:FiniteCc}, 
the statement in Theorem \ref{T:MomComSDE1} is true with $\bbF[C_{p,+}^{2,v}]$, $\bbF[C^{2,v}_{b,-}]$, and $\bbF[C^{2,v}_p]$ replaced by $\mathbb{F}[C_{+}^{2,v}]$, $\mathbb{F}[C_{-}^{2,v}]$ and $\mathbb{F}[C^{2,v}]$, respectively.
\end{theorem}
 
\begin{theorem}\label{T:FiniteCc2}
Under Assumption \ref{A:FiniteCc}, 
the statement in Theorems \ref{T:MomComSDE2} is true with $\bbF[C_{p,+}^{2,v}]$, $\bbF[C^{2,v}_{b,-}]$, and $\bbF[C^{2,v}_p]$ replaced by $\mathbb{F}[C_{+}^{2,v}]$, $\mathbb{F}[C_{-}^{2,v}]$ and $\mathbb{F}[C^{2,v}]$, respectively.
\end{theorem}

Both Theorem \ref{T:MomComSDE1} and Theorem \ref{T:FiniteCc} are essentially covered by Cox, Fleischmann and Greven \cite{CFG96}. The main difference is that Theorem \ref{T:MomComSDE1} covers a much richer family of functions and another difference is that we have correlated, instead of independent, Brownian motions; See Remark \ref{R:SDEs} below for more details. 

\bigskip 
  
\subsection{Outline of the paper}
This paper is organized as follows: After some definitions, notation and preliminaries in Section \ref{S:Def}, we provide the approximation procedure which shows that SHE \eqref{E:SHE} with rough initial data and noise whose spatial correlation only satisfies Dalang's condition \eqref{E:Dalang} can be approximated by systems of infinite dimensional SDEs (i.e., interacting diffusions on the $\ud$-dimensional lattice) in Section \ref{S:Aprox}. 
Combining the approximation procedures and the comparison theorems for infinite dimensional SDEs (Theorems \ref{T:MomComSDE1} and \ref{T:MomComSDE2}, and Corollary \ref{C:Slepian}) proves the main theorems \ref{T:MomComp} and \ref{T2:MomComp} and also Slepian's inequality for SPDEs -- Corollary \ref{C:SPDE:Slepian} in Section \ref{SS:Step4}. 
It remains to establish Theorems  \ref{T:MomComSDE1} and \ref{T:MomComSDE2}, and Corollary \ref{C:Slepian}, which is done in Section \ref{S:MomComSDE}. We first prove the existence and uniqueness result --- Theorem \ref{T:ExtUniqSDE1} --- in Section \ref{SS:Existence}. Then we will prove Theorems \ref{T:MomComSDE1} and \ref{T:MomComSDE2} by first showing that a system of infinite dimensional SDEs can be approximated by systems of finite dimensional SDEs with a nice $\rho$ in Section \ref{SS:Aprox2} and then obtaining the comparison theorems for finite dimensional SDEs following the procedure of Cox, Fleischmann and Greven \cite{CFG96} in Section \ref{SS:Finite}.
With these preparations, we proceed to prove 
Theorems \ref{T:MomComSDE1} and \ref{T:MomComSDE2} and Corollary \ref{C:Slepian} in Section \ref{SS:4.4}.
Finally, in Section \ref{S:ExApp}, we give several examples to cover those in (E-1) -- (E-5) above and one application of our approximation results to give another straightforward proof for the weak sample path comparison principle.

\section{Some definitions, notation and preliminaries}\label{S:Def}
Throughout this paper, $\Norm{\cdot}_p$ denotes the $L^p(\Omega)$-norm, $\bbN:=\{0,1,2,\cdots\}$, $\Lip_\rho$ refers to the Lipschitz constant for $\rho$,  
$D_i:=\frac{\partial}{\partial x_i}$, and $\R_+:=[0,\infty)$.

\bigskip
Recall that a {\em spatially homogeneous Gaussian noise that is white in time} is an
$L^2(\Omega)$-valued mean zero Gaussian process on a complete probability space $\left(\Omega,\calF,\bbP\right)$
\[
\left\{F(\psi):\: \psi\in C_c^{\infty}\left([0,\infty)\times\R^{d}\right)\:\right\},
\]
such that
\[
\E\left[F(\psi)F(\phi)\right] = \int_0^{\infty} \ud s\iint_{\R^{2d}}\psi(s,x)\phi(s,y)f(x-y)\ud x\ud y.
\]
Let $\calB_b(\R^d)$ be the collection of Borel measurable sets with finite Lebesgue measure.
As in Dalang-Walsh theory \cite{Dalang99,Walsh},
one can extend $F$ to a $\sigma$-finite $L^2(\Omega)$-valued martingale measure $B\mapsto F(B)$
defined for $B\in \calB_b(\R_+\times\R^d)$, where $\R_+:=[0,\infty)$. Then define
\[
M_t(B) :=F\left([0,t]\times B \right), \quad B\in\calB_b(\R^d).
\]
Let $(\calF_t,t\ge 0)$ be the natural filtration generated by $M_\cdot(\cdot)$
and augmented by all $\bbP$-null sets $\calN$ in $\calF$, i.e.,
\[
\calF_t := \sigma\left(M_s(A):\: 0\le s\le t,
A\in\calB_b\left(\R^d\right)\right)\vee
\calN,\quad t\ge 0,
\]
Then for any adapted, jointly measurable (with respect to
$\calB\left((0,\infty)\times\R^d\right)\times\calF$) random field $\{X(t,x): t>0,x\in\R^d\}$ such that for all integers $p\ge 2$,
\[
\int_0^\infty\ud s\iint_{\R^{2d}}\ud x\ud y\:
\Norm{X(s,y)X(s,x)}_{\frac{p}{2}} f(x-y) <\infty,
\]
the stochastic integral
\[
\int_0^\infty \int_{\R^d} X(s,y)M(\ud s,\ud y)
\]
is well-defined in the sense of Dalang-Walsh. Here we only require the joint-measurability instead of
predictability; see  in \cite[Proposition 2.2]{CK15SHE} and \cite[Proposition 3.1]{ChenDalang13Heat}.

Let $J_0(t,x)$ denote the solution to the homogeneous equation 
\begin{align}
 \label{E:J0}
J_0(t,x) := (\mu * G(t,\cdot))(x) =\int_{\R^d} G(t,x-y)\mu(\ud y),
\end{align}
and $I(t,x)$ be the stochastic integral in the mild form \eqref{E:mild}.
Hence, the mild form \eqref{E:mild} can be written as $u(t,x)=J_0(t,x) + I(t,x)$.

\begin{definition}\label{D:Solution}
A process $u=\left(u(t,x),\:(t,x)\in(0,\infty)\times\R^d \right)$  is called a {\it
random field solution} to \eqref{E:SHE} if
\begin{enumerate}[(1)]
 \item $u$ is adapted, i.e., for all $(t,x)\in(0,\infty)\times\R^d$, $u(t,x)$ is
$\calF_t$-measurable;
\item $u$ is jointly measurable with respect to
$\calB\left((0,\infty)\times\R^d\right)\times\calF$;
\item $\Norm{I(t,x)}_2<+\infty$ for all $(t,x)\in(0,\infty)\times\R^d$;
\item  $I$ is $L^2(\Omega)$-continuous, i.e., the function $(t,x)\mapsto I(t,x)$ mapping $(0,\infty)\times\R^d$ into
$L^2(\Omega)$ is continuous;
\item $u$ satisfies \eqref{E:mild} a.s.,
for all $(t,x)\in(0,\infty)\times\R^d$.
\end{enumerate}
\end{definition}

Existence and uniqueness of a random field solution for bounded initial data is covered by classical Dalang-Walsh theory \cite{Dalang99,Walsh}. For rough initial data, this is established in \cite{ChenDalang13Heat,CH18Comparison, CK15SHE,Huang}. A key tool for dealing the rough initial data is the following moment formula. We need first introduce some notation.
Denote
\begin{align}\label{E:k}
k(t):=\int_{\R^d}f(z)G(t,z)\ud z.
\end{align}
By the Fourier transform, this function can be written in the following form
\begin{align}\label{E:k2}
k(t):=(2\pi)^{-d} \int_{\R^d}\hat{f}(\ud\xi)\exp\left(-\frac{t|\xi|^2}{2}\right).
\end{align}
Define $h_0(t):=1$ and for $n\ge 1$,
\begin{align}\label{E:hn}
 h_n(t)= \int_0^t \ud s \: h_{n-1}(s) k(t-s).
\end{align}
% Let
% 
% This function is defined through the correlation function $f$. The following lemma
% tells us that this function has an exponential bound.

\begin{theorem}[Moment bounds, Theorem 1.7 of \cite{CH18Comparison}]
\label{T:Mom}
Under Dalang's condition \eqref{E:Dalang},
if the initial data $\mu$ is a signed measure that satisfies \eqref{E:J0finite},
then the solution $u$ to \eqref{E:SHE} for any given $t>0$ and $x\in\R^d$ is in $L^p(\Omega)$, $p\ge 2$, and
\begin{align}\label{E:Mom}
 \Norm{u(t,x)}_p \le
 \left[\Vip+\sqrt{2} \left(|\mu|*G(t,\cdot)\right)(x) \right]H\left(t;\gamma_p\right)^{1/2},
\end{align}
where $\Vip=|\rho(0)|/\Lip_\rho$, $\gamma_p=32p\Lip_\rho^2$, $\LIP_\rho>0$ is the Lipschitz constant for $\rho$, and 
\begin{align}\label{E:H}
 H(t;\gamma):= \sum_{n=0}^\infty \gamma^n h_n(t),\qquad\text{for all $\gamma\ge 0$.}
\end{align}
Moreover, if the {\em strengthened Dalang's condition} \eqref{E:DalangAlpha} is satisfied, namely, 
\begin{align}\label{E:DalangAlpha}
\int_{\R^d}\frac{\hat{f}(\ud\xi)}{\left(1+|\xi|^2\right)^{1-\alpha}}<\infty, \quad\text{for some $\alpha\in(0,1]$,}
\end{align}
then when $p\ge 2$ is large enough, there exists some constant $C>0$ such that 
\begin{align}\label{E:MomAlpha}
\Norm{u(t,x)}_p \leq C \Big[\: \Vip+\left(|\mu|*G(t,\cdot)\right)(x) \Big] \exp\left(C \:\Lip_\rho^{2/\alpha} p^{1/\alpha} t\right)\,.
\end{align}
% for all $t>0$ and $x\in\R^d$.
\end{theorem}

Note that $H(t;\gamma)$ in \eqref{E:H} has genuine exponential growth as proved in the following lemma:

\begin{lemma}[Lemma 2.5 in \cite{CK15SHE} or Lemma 3.8 in \cite{BC16}]
\label{L:EstHt}
For all $t\ge 0$ and $\gamma\ge 0$, recalling that $\Upsilon(\beta)$ is defined in \eqref{E:Dalang}, it holds that
\begin{align}
\label{E:Var} 
\limsup_{t\rightarrow\infty} \frac{1}{t}\log H(t;\gamma)\le \inf\left\{\beta>0:  \:\Upsilon\left(2\beta\right) < \frac{1}{2\gamma}\right\}.
\end{align}
\end{lemma}

%\subsection{Two approximation results (Proof of Theorem \ref{T:Approx})}
%\label{S:Approx}

\section{Approximation procedure and the proof of Theorems \ref{T:MomComp} and \ref{T2:MomComp}}
\label{S:Aprox}

The following approximation procedure is interesting by itself, based on which our comparison results are direct consequences (see Step 4). Basically, we show that stochastic heat equations on $\R^d$ with rough initial condition and driven by Gaussian noise which is white in time and correlated in space can be approximated by systems of interacting diffusions on the $d$-dimensional lattice. There are several steps in order to achieve this goal. 

\subsection{Step 1 (Regularization of the initial data and noise)}
We will need the following approximation results, which were proved in Theorem 1.9 of \cite{CH18Comparison} for $L^2(\Omega)$ case. The generalization to the $L^p(\Omega)$, $p\ge 2$, is straightforward thanks to the moment formula \eqref{E:Mom}.

\begin{lemma}
\label{L:Approx}
{Assume that $f$ satisfies Dalang's condition \eqref{E:Dalang}.}

{\noindent (1)} Suppose that the initial measure $\mu$ satisfies \eqref{E:J0finite}.
If $u$ and $u_{\epsilon}$ are the solutions to \eqref{E:SHE} starting from
$\mu$ and
$((\mu\:\psi_\epsilon)*G(\epsilon,\cdot))(x)$, respectively, where
\begin{align}\label{E:psi}
 \psi_\epsilon(x) = \one_{\{|x|\le 1/\epsilon\}} + \left(1+1/\epsilon - |x| \right)\one_{\{1/\epsilon<|x|\le 1+1/\epsilon\}},
\end{align}
then
\[
\lim_{\epsilon\rightarrow 0_+}\Norm{u(t,x)-u_{\epsilon}(t,x)}_p =0,\quad \text{for all $p\ge 2$, $t>0$ and $x\in\R^d$.}
\]
{\noindent (2)} 
Let $\phi$ be any continuous, nonnegative and nonnegative definite function on $\R^d$ 
with compact support such that $\int_{\R^d}\phi(x)\ud x=1$. 
Let $u$ be the solution to \eqref{E:SHE} starting from bounded initial data,
i.e., $\mu(\ud x)=g(x)\ud x$ with $g\in L^\infty(\R^d)$.
If $\tilde{u}_\epsilon$ is the solution to the following mollified equation 
\begin{equation}\label{E:SHE regularize}
\frac{\partial}{\partial t} \tilde{u}_\epsilon(t,x) =\frac{1}{2}\Delta \tilde{u}_\epsilon (t,x)+\rho(\tilde{u}_\epsilon (t,x))\dot{M}^{\epsilon}(t,x)\,,
\end{equation}
with the same initial condition $\tilde{u}_\epsilon(0,\cdot)=\mu$ as $u$, 
where 
\begin{equation}
M^{\epsilon}(\ud s, \ud x) = \int_{\RR^d} \phi_{\epsilon}(x-y) M(\ud s,\ud y)\ud x \,, 
\end{equation}
and $\phi_\epsilon(x)=\epsilon^{-d}\phi(x/\epsilon)$,
then the spatial correlation function $f^{\epsilon,\epsilon}$ for $M^\epsilon$  is given by $f^{\epsilon, \epsilon}=\phi_\epsilon * \phi_\epsilon * f$ which satisfies the strengthened Dalang's condition \eqref{E:DalangAlpha} with $\alpha=1$ and
\begin{align}
 \lim_{\epsilon\rightarrow 0_+}\sup_{x\in\R^d}\Norm{u(t,x)-\tilde{u}_{\epsilon}(t,x)}_p =0,\quad 
 \text{for all $p\ge 2$ and $t>0$.}
\end{align}
Moreover, one can always find such $\phi$ so that 
\begin{align}\label{E:f'(0)=0}
&f^{\epsilon,\epsilon}(\cdot)\in C^2(\R^d;\R_+) \quad \text{with} \quad \frac{\partial }{\partial x_i}
f^{\epsilon,\epsilon}(0) =0 \quad \text{and}\quad 
\sup_{x\in\R^d} \left|\frac{\partial^2 }{\partial x_i x_j}
f^{\epsilon,\epsilon}(x) \right|<\infty
\end{align}
for all $i,j=1,\cdots, d$.
\end{lemma}

\begin{proof} 
{\bf\noindent (1)~}
Theorem 1.7 of \cite{CH18Comparison} shows that 
\begin{align}\label{E_:L2Conv}
u_\epsilon (t,x) \rightarrow u(t,x)\quad \text{in $L^2(\Omega)$, for all $t>0$ and $x\in\R^d$.}
\end{align}
If one can show that for any $p> 2$
\begin{align}\label{E_:UnifLp}
\sup_{\epsilon\in (0,1)} \Norm{u_\epsilon(t,x)}_p<\infty,\quad\text{for all $t>0$ and $x\in\R^d$,}
\end{align}
then the $L^2(\Omega)$ convergence in \eqref{E_:L2Conv} also holds for all $p>2$.

Let $\mu_\epsilon:=\left((\mu \: \psi_\epsilon)*G(\epsilon,\cdot)\right)(x)$. Since, for some constant $C_t>0$, $G(t+\epsilon,x)\le C_t G(2t,x)$ for all $x\in\R^d$ and $\epsilon\in (0,1\wedge t)$, we have 
\begin{align*}
\left(|\mu_\epsilon|*G(t,\cdot)\right)(x)
\le \left(|\mu| *G(t+\epsilon,\cdot)\right)(x) \leq C_t  \left(|\mu| *G(2t,\cdot)\right)(x)<\infty.
\end{align*}
Hence, \eqref{E:Mom} in Theorem \ref{T:Mom} shows \eqref{E_:UnifLp}, which also  proves part (1) of Theorem \ref{L:Approx}.

{\bigskip\bf\noindent (2)~}
From the direct computation,  the spatial covariance function for $M^\epsilon$ is given by $f^{\epsilon,\epsilon}(x):= (\phi_\epsilon*\phi_\epsilon*f)(x)$ and  $f^{\epsilon,\epsilon}$ is nonnegative and nonnegative definite and it satisfies strengthened Dalang's condition \eqref{E:DalangAlpha} for each $\epsilon>0$ (see Step 3 of Section 7 of \cite{CH18Comparison}).  In addition, the $L^2(\Omega)$ convergence has been established in Theorem 1.7 of \cite{CH18Comparison}. As in the proof of part (1), we need only show that for all $t>0$ and $p\ge 2$,
\begin{align}
\label{E_:UnifLp2}
\sup_{\epsilon\in (0,1)} \sup_{x\in\R^d}\Norm{\tilde{u}_\epsilon(t,x)}_p<\infty.
\end{align}
By Theorem \ref{T:Mom}, 
\[
\Norm{\tilde{u}_\epsilon (t,x)}_p\le C  H_\epsilon (t,\gamma_p)^{1/2},
\]
where we have used the fact that $g\in L^\infty(\R^d)$ and $\gamma_p=32 p\Lip_\rho^2$, and $H_\epsilon(t;\gamma_p)$ is defined in \eqref{E:H} with the function $k(\cdot)$ replaced by $k_\epsilon(\cdot)$.
Because 
\begin{align*}
% \label{E:hatphie}
|\hat{\phi}_\epsilon(\xi)|^2=
|\hat{\phi}(\epsilon\xi)|^2
=\left|\int_{\R^d}e^{-i\epsilon\InPrd{\xi,x}}\phi(x)\ud x\right|^2
\le \left(\int_{\R^d}\phi(x)\ud x\right)^2 =1,
\end{align*}
which implies that 
\begin{align*}
k_\epsilon(t)&=\int_{\R^d}f^{\epsilon,\epsilon}(z) G(t,z)\ud z
=(2\pi)^{-d}\int_{\R^d}\hat{f}(\ud \xi) \hat{\phi}_\epsilon(\xi)^2 \exp\left(-\frac{t|\xi|^2}{2}\right)\\
&\le
(2\pi)^{-d}\int_{\R^d}\hat{f}(\ud \xi) \exp\left(-\frac{t|\xi|^2}{2}\right)= k(t),
\end{align*}
for all $\epsilon>0$, where $K(t)$ is defined in \eqref{E:k},
we see that 
\[
\Norm{\tilde{u}_\epsilon (t,x)}_p\le C  H (t,\gamma_p)^{1/2},
\]
where the upper bound is uniform in both $\epsilon$ and $x$. 
 
\bigskip
It remains to prove \eqref{E:f'(0)=0}. 
Let $g(x)=1_{[-1,1]^d}(x)$ for $x\in\R^d$ and choose
\[
\phi(x) =4^{-d}(g*g)(x) = 4^{-d} \prod_{i=1}^d (2-|x_i| ) \one_{\{|x_i|\le 2\}}.
\]
It is easy to see that  $\phi$ is a continuous, nonnegative and nonnegative definite function on $\R^d$ 
with compact support such that $\int_{\R^d}\phi(x)\ud x=1$. It is also clear that $f^{\epsilon,\epsilon}(\cdot)\in C^2(\R^d;\R_+)$. 
Fix $i\in\{1,\cdots,d\}$. 
To show that $\frac{\partial }{\partial x_i}f^{\epsilon,\epsilon}(0)=0$, it suffices to show this for $\epsilon =1$. 
Direct computation shows that 
\[
\frac{\partial^n}{\partial x_i^n}f^{1,1}(0)=
\int_{\R^d} \frac{\partial^n}{\partial x_i^n}\phi_2(y) \:  f(\ud y)
\]
with  $\phi_2(x) = (\phi*\phi)(x)  = 4^{-d} \prod_{i=1}^d \theta(x_i)$ and 
\begin{gather*}
\theta(x_i) = 
 \left(\frac{1}{2}(|x_i|-4)x_i^2+\frac{16}{3}\right)
\one_{\{|x_i|\le 2\}}+
\frac{1}{6}
(4-|x_i|)^3 \one_{
\{2\le |x_i|\le 4\}}.
\end{gather*}
It is clear that $\theta(\cdot)$ is an even and $C^2$ function on $\R$ with $\theta'(\cdot)$ being a continuous odd function and $\theta''(\cdot)$ a continuous even function. More precisely,  
\begin{align*}
\theta'(x_i) &= 
\frac12 x_i (3 |x_i|-8)\one_{\{|x_i|\le 2\}}
-\frac12 \sign(x_i) (4-|x_i|)^2\one_{\{2\le |x_i|\le 4\}},\\
\theta''(x_i) &= 
 (3 |x_i|-4)\one_{\{|x_i|\le 2\}}
- (4-|x_i|)\one_{\{2\le |x_i|\le 4\}}.
\end{align*}
Because $f$ is nonnegative definite, we see that 
\[
\frac{\partial}{\partial x_i}f^{1,1}(0) = \int_{\R^d}  \frac{\partial}{\partial x_i} \phi_2(y) f(\ud y)=
4^{-d} \int_{\R^d} \left(\prod_{j\ne i} \theta(y_j) \right)\theta'(y_i)f(\ud y)=0.
\]
One can also find some constant $C>0$ large enough such that for all $i,j\in\{1,\cdots,d\}$, 
\[
\left|\frac{\partial^2}{\partial x_i x_j}\phi_2(x)\right| \le C \prod_{i=1}^d \left(4-|x_i|\right) \one_{\{|x_i|\le 4\}}
= C \left(\one_{\{|\cdot|\le 2\}} *\one_{\{|\cdot|\le 2\}}\right)(x),
\]
where the right-hand side is a continuous, nonnegative, and nonnegative definite function.
Hence, for any $i,j\in\{1,\cdots,d\}$, 
\begin{align*}
\left|\frac{\partial^2}{\partial x_i x_j}f^{1,1}(x)\right| \le 
& \int_{\R^d}  
\left|\frac{\partial^2}{\partial x_ix_j}\phi_2(x-y)\right|f(\ud y)\\
\le  & 
C \int_{\R^d} \prod_{i=1}^d (4-|x_i-y_i|)\one_{\{|x_i-y_i|\le 4\}}f(\ud y)\\
\le& 
C \int_{\R^d} \prod_{i=1}^d (4-|y_i|)\one_{\{|y_i|\le 4\}}f(\ud y)< \infty,
\end{align*}
where the third inequality is due to the fact that the integrand is a nonnegative definite function and in the last inequality we use the fact that the integrand is a continuous function. 
This completes the proof of Lemma \ref{L:Approx}. 
\end{proof}

%The above theorem is a slight generalization of the $L^2(\Omega)$ convergence of Theorem 1.7 in \cite{CH18Comparison} to the $L^p(\Omega)$ convergence. A sketch of the proof of this theorem is given in Section \ref{S:Approx}.

We also point out that the initial data $((\mu\:\psi_\epsilon)*G(\epsilon,\cdot))(x)$ in the above theorem has Gaussian tails so that \textcolor{magenta}{it} is in $L^p(\R^d)$ for any $p\in [1,\infty]$. This will be used in Step 4 of Section \ref{SS:Step4}. 

\begin{lemma}\label{L:initial} 
 Suppose that $\mu$ is a (possibly signed) Borel measure that satisfies \eqref{E:J0finite}.
 For any $\delta>\epsilon>0$, 
 there exists some constant $C=C(\epsilon,\delta,\mu)>0$ such that 
 \[
 \left|((\mu\:\psi_\epsilon)*G(\epsilon,\cdot))(x)\right|
 \le C\: G(\delta, x), 
 \]
 for all $x\in\R^d$, where $\psi_\epsilon(\cdot)$ is given by \eqref{E:psi}. 
\end{lemma}
\begin{proof}
Fix $\delta>\epsilon>0$ and denote
\[
\Psi(x):=\frac{|((\mu\:\psi_\epsilon)*G(\epsilon,\cdot))(x)|}{G(\delta,x)}.
\]
It is clear that $\Psi$ is a nonnegative and smooth function. Notice that 
\[
\frac{G(\epsilon,x-y)}{G(\delta,x)}
=(\delta/\epsilon)^{d/2}
\exp\left(-\frac{(\delta-\epsilon)\left|x-\frac{\delta}{\delta-\epsilon}y\right|^2}{2\epsilon\delta} + \frac{|y|^2}{2(\delta-\epsilon)}\right),
\]
which implies that 
\begin{align*}
 \int_{\R^d}\Psi(x)\ud x 
 &\le 
(\delta/\epsilon)^{d/2} 
\int_{|y|\le 1+1/\epsilon} |\mu|(\ud y)
\exp\left(\frac{|y|^2}{2(\delta-\epsilon)}\right)
\int_{\R^d}\ud x 
\exp\left(-\frac{(\delta-\epsilon)\left|x-\frac{\delta}{\delta-\epsilon}y\right|^2}{2\epsilon\delta}\right)\\
&\le 
(\delta/\epsilon)^{d/2} 
\exp\left(\frac{(1+1/\epsilon)^2}{2(\delta-\epsilon)}\right)
\left(\int_{|y|\le 1+1/\epsilon} |\mu|(\ud y)\right)
\int_{\R^d}\ud x 
\exp\left(-\frac{(\delta-\epsilon)\left|x\right|^2}{2\epsilon\delta}\right)\\
&= C_{\epsilon,\delta,\mu}.
\end{align*}
Hence, $\Psi(x)$ is also an integrable function on $\R^d$. These facts together imply that $\Psi(x)$ has to be bounded. This proves the lemma. 
\end{proof}

\subsection{Step 2 (Mollification of the Laplacian operator)}

In this section, we regularize the Laplacian operator by using Yosida's approximation. 
Thanks to Step 1, we may assume that the initial data $\mu(\ud x)=u_0(x)\ud x$ with $u_0\in \mathcal{S}(\R^d)$, i.e., $u_0$ is a Schwartz test function, and $f\in C^2 (\R^d;\R_+)$ with
\begin{align}\label{E:fee}
\frac{\partial }{\partial x_i}
f^{\epsilon,\epsilon}(0) =0 \quad \text{and}\quad 
\sup_{x\in\R^d} \left|\frac{\partial^2 }{\partial x_i x_j}
f^{\epsilon,\epsilon}(x) \right|<\infty,\quad\text{for all $i,j=1,\cdots,d$.}
\end{align} 
First, let us view the $G(t,x)$ as an operator, denoted by ${\bf G}(t)$, as follows:
\begin{equation}
{\bf G}(t) f(x) := (G(t,\cdot)* f)(x)\,.
\end{equation}
Let $\bf I $ be the identity operator: ${\bf I} f(x):=(\delta * f)(x) = f(x)$. For any $\epsilon\in(0,1)$, set
\begin{equation}
\Delta ^{\epsilon} = \frac{{\bf G}(\epsilon) - {\bf I}}{\epsilon}\,.
\end{equation}
Let
\begin{equation}
G^{\epsilon}(t) = \exp (t\Delta^{\epsilon})= e^{-\frac{t}{\epsilon}} \sum_{n=0}^{\infty} \frac{(t/\epsilon)^n}{n!} {\bf G}(n\epsilon) := e^{-t/\epsilon} {\bf I} + {\bf R}^{\epsilon}(t)\,,
\end{equation}
where the operator ${\bf R}^{\epsilon}(t)$ has a density, denoted by $R^{\epsilon}(t,x)$, which is equal to
\begin{equation}\label{E:R epsilon}
R^{\epsilon}(t,x)=e^{-t/\epsilon} \sum_{n=1}^{\infty} \frac{(t/\epsilon)^n}{n!} G(n\epsilon,x)\,.
\end{equation}
Because $f\in C^2(\R^d;\R_+)$, the stochastic integral with respect to $M(\ud s,\ud y)$ is equivalent to the stochastic integral with respect to $M_y(\ud s)\ud y$, where $\{M_x(t),t\ge 0, x\in\R^d \}$ are Brownian motions starting from zero indexed by $x\in\R^d$ with the following correlation structure
\begin{align}
\label{E:d<M>t} 
\E[M_x(t)M_y(t)] = f(x-y) \: t .
\end{align}
Denote $\dot{M}_x(t) = \frac{\ud}{\ud t} M_x(t)$.
Consider the following stochastic differential equation
\begin{equation}\label{E:apprx}
 \left\{
\begin{array}{ll}
\displaystyle     \frac{\partial }{\partial t} u_{\epsilon}(t,x) = 
\Delta^{\epsilon} u_{\epsilon}(t,x) + \rho(u_{\epsilon}(t,x)) 
\dot{M}_{x}(t)\,, & t >0\,, x \in \RR^d\,,
\\[1em]
\displaystyle      u_{\epsilon}(0,x) = u_0(x)\,, & x \in \RR^d \,. \\
\end{array}
\right.
\end{equation}
Since $\rho$ is Lipschitz continuous and $\Delta^{\epsilon}$ is a bounded operator,
\eqref{E:apprx} has a unique strong solution
\begin{equation}\label{E:SSol1}
u_{\epsilon}(t,x)=
u_0(x) + \int_0^t \ud s  \Delta^{\epsilon} 
u_{\epsilon}(s,x)
+ \int_0^t \rho(u_{\epsilon}(s,x))  M_x(\ud s)\,,
\end{equation}
where
\begin{align}\label{E_:SSol1}
\int_0^t \ud s  \Delta^{\epsilon} u_{\epsilon}(s,x) = 
\frac{1}{\epsilon}
\int_0^t\int_{\R^d} G(\epsilon,x-y)\left[u_\epsilon(s,y)- u_\epsilon(s,x)\right]\ud y \ud s.
\end{align}

We will need the following lemma regarding the spatial regularity of $u_\epsilon(t,x)$.

\begin{lemma}\label{L:HolderXeps}
Let $u_\epsilon$ be a solution to \eqref{E:apprx}.
If the initial data $u_0\in \mathcal{S}(\R^d)$, and if the correlation function $f$ in \eqref{E:d<M>t} is in $\in C^2(\R^d;\R_+)$ with $f'(0)\equiv 0$ and $f''(\cdot)$ being bounded, then for any $\epsilon>0$, $T>0$, and $p\ge 2$, there is a constant $C=C(T,p,\epsilon, \mu,\Lip_\rho)>0$ such that 
\[
\Norm{u_\epsilon(t,x)-u_\epsilon(t,y)}_p
\le C |x-y|,\qquad\text{
for all $t\in[0,T]$ and $x,y\in\R^d$.}
\]
for all $t\in[0,T]$ and $x,y\in\R^d$ with $|x-y|\le K$.
\end{lemma}
\begin{proof}
Fix $p\ge 2$, $T>0$, and $\epsilon>0$. 
Let $C$ be a generic constant that may depend on these constants, namely, $T$, $p$, $\epsilon$, and $\Lip_\rho$. 
For any $t\in [0,T]$ and $x,y\in\R^d$, we have that 
\begin{align*}
u_\epsilon(t,x)-u_\epsilon(t,y) = & 
u_0(x) -u_0(y) \\
&+ \frac{1}{\epsilon}\int_0^t \ud s \int_{\R^d}
\ud z
\left[G(\epsilon,x-z)-G(\epsilon,y-z)\right] 
\left[u_\epsilon(s,z)-u_\epsilon(s,x)\right]\\
&+\frac{1}{\epsilon}\int_0^t\ud s 
\left[u_\epsilon(s,y)-u_\epsilon(s,x)\right] \\
&+\int_0^t\left[\rho(u_\epsilon(s,x))-\rho(u_\epsilon(s,y))\right] M_x(\ud s)\\
&+\int_0^t\rho(u_\epsilon(s,y))\left[
 M_x(\ud s)- M_y(\ud s)\right]\\
=:& \sum_{\ell=1}^5 I_\ell.
\end{align*}
It is clear that $|I_1|\le \Norm{u_0'}_{L^\infty(\R^d)} |x-y|$.
The boundedness and regularity of the initial data implies that
\begin{align}\label{E:ATp}
A_{T,p,\epsilon}:=\sup_{s\in [0,T]}\sup_{x\in\R^d}
\Norm{u_\epsilon(s,x)}_p<\infty.
\end{align}
Hence, we have that
\begin{align*}
\Norm{I_2}_p\le & 2 A_{T,p,\epsilon} \frac{1}{\epsilon}\int_0^t\ud s
\int_{\R^d}\ud z \left|G(\epsilon,x-z)-G(\epsilon,y-z)\right|.
\end{align*}
Notice that Lemma 3.1  of \cite{CH18Comparison}  with $\alpha=1$ implies that 
\[
\left|G(\epsilon,x-z)-G(\epsilon,y-z)\right| 
\le \frac{C}{\sqrt{\epsilon}}\left(G(2\epsilon,x-z)+G(2\epsilon,y-z)\right)|x-y|. 
\]
Therefore, 
\[
\Norm{I}_p\le  C|x-y|\int_0^t \ud s\: \int_{\R^d} \ud z \left(G(2\epsilon,x-z)+G(2\epsilon,y-z)\right)
= C|x-y|,
\]
where we note that the constants $C$ depend on $\epsilon$.

As for $I_3$, we see that
\[
\Norm{I_3}_p^2 \le 
\frac{1}{\epsilon^2}\left(\int_0^t \Norm{u_\epsilon(s,y)-u_\epsilon(s,x)}_p\ud s\right)^2
\le \frac{T}{\epsilon^2}
\int_0^t \Norm{u_\epsilon(s,y)-u_\epsilon(s,x)}_p^2\ud s.
\]
As for $I_4$, by \eqref{E:d<M>t} and the Burkholder-Davis-Gundy inequality, we see that 
\begin{align}\label{E:I4BDG}
 \Norm{I_4}_p^2 \le & 
 C \Lip_\rho  f(0) \int_0^t \Norm{u_\epsilon(s,y)-u_\epsilon(s,x)}_p^2\ud s.
\end{align}
As for $I_5$, by the Burkholder-Davis-Gundy inequality, \eqref{E:d<M>t}, and \eqref{E:ATp}, we see that
% \begin{align*}
% \Norm{I_5}_p^2 \le C \int_0^t\ud s \iint_{\R^{2d}} \ud z\ud z'
% &\left|G(\epsilon,x-z)-G(\epsilon,y-z)\right|\\
% \times &\left|G(\epsilon,x-z')-G(\epsilon,y-z')\right|f(z-z').
% \end{align*}
\begin{align*}
\Norm{I_5}_p^2 \le & C \Norm{\InPrd{\int_0^\cdot \rho(u_\epsilon(s,y)) \left[M_x(\ud s)-M_y(\ud s)\right]}_t}_{p/2}\\
= & C \Norm{2 \int_0^t \rho(u_\epsilon(s,y))^2  \left[f(0)-f(x-y)\right]\ud s}_{p/2}\\
\le & C \int_0^t \Norm{\rho(u_\epsilon(s,y))}_{p}  \left|f(0)-f(x-y)\right|\ud s\\
\le & C \left|f(0)-f(x-y)\right|.
\end{align*}
Because $f\in C^2(\R^d;\R_+)$ with properties \eqref{E:fee}, we see that $|f(0)-f(x-y)|\le C |x-y|^2$.  Therefore, 
\[
\Norm{I_5}_p\le C|x-y|.
\]
Combining these five terms, we see that
\[
\Norm{u_\epsilon(t,x)-u_\epsilon(t,y)}_p^2 
\le C|x-y|^{2}+ C \int_0^t \Norm{u_\epsilon(s,y)-u_\epsilon(s,x)}_p^2\ud s.
\]
Finally, an application of Gronwall's lemma proves Lemma \ref{L:HolderXeps}. 
\end{proof}

\begin{lemma}\label{L:ApproxLaplacian}
Under the same setting as in Lemma \ref{L:HolderXeps}, we have that 
\begin{equation}\label{eq: appro sol L2 conv}
\lim_{\epsilon \to 0} \sup_{x \in \RR^d} \|u_{\epsilon}(t,x)-u(t,x)\|_p = 0\,, \quad\text{for all $t>0$ and $p\ge 2$}\,,
\end{equation}
where $u(t,x)$ is the solution to the same equation \eqref{E:apprx} but with $\Delta^\epsilon$ replaced by the standard Laplacian operator $\Delta$.
\end{lemma}

\begin{proof}
The proof of this lemma is very similar to the proof of Lemma \ref{L:Approx}. First of all, (7.9) in Step 2 of Section 7 in \cite{CH18Comparison} shows that for any $t>0$,
\[ \lim_{\epsilon\to 0^+} \sup_{x\in\R^d} \|u_\epsilon(t, x) - u(t, x)\|_2 =0.\]
Note that in Step 2 of Section 7 of \cite{CH18Comparison}, one mollifies the Laplacian, initial data and the noise at the same time. Here, we do that in separate steps. The arguments in this slightly simplified case can be carried out line-by-line. We leave the details for the interested readers. 
Then, the boundedness of initial data implies that 
\[ \sup_{\epsilon\in(0,1)}\sup_{s\in [0,T]}\sup_{x\in\R^d}
\Norm{u_\epsilon(s,x)}_p<\infty,\]
which  basically implies \eqref{eq: appro sol L2 conv}
\end{proof}

\subsection{Step 3 (Discretization in space)}
For $\delta\in (0,1)$ and $x\in\R$, denote 
\[
[x]_\delta:=\left[\frac{x}{\delta}\right] \: \delta,
\]
where $[x]$ is the function that rounds half away from zero, e.g., $[4.5]=5$ and 
$[-4.5]= -5$. Note that $[x]_\delta$ is an odd function of $x$. 
Moreover, for $x\in\R^d$, set
\[
[x]=([x_1],\dots, [x_d])\quad\text{and}\quad
[x]_\delta=([x_1]_\delta,\dots, [x_d]_\delta).
\] 
For $x\in\R^d$, denote 
\[
Q_\delta (x) := \left\{y\in\R^d: \: |x_i-y_i|\le \frac{\delta}{2}, \: i=1,\dots, d\right\}.
\]
For $\epsilon$ and $\delta\in (0,1)$, and $i,j\in \bbZ^d$, let 
\[
P_{ij}^{\epsilon,\delta} :=\int_{Q_\delta(j \delta)} G_d(\epsilon,i\delta -y)\ud y
=\prod_{k=1}^d \int_{(i_k-j_k-1/2)\delta}^{(i_k-j_k+1/2)\delta}
G_1(\epsilon,y)\ud y,
\]
where $G_d(t,x)$ is the heat kernel on $\R^d$ (see \eqref{E:HeatKernel}) and when there is no confusion from the context, we will simply write it as $G(t,x)$.

Now we consider the following infinite dimensional SDE: 
\begin{equation}\label{E:apprx2}
\begin{cases}
\displaystyle      
\ud u_{\epsilon}^\delta(t,i\delta) = 
\frac{1}{\epsilon}\sum_{j\in\bbZ^d} P_{ij}^{\epsilon,\delta} \left[ u_{\epsilon}^\delta(t,j\delta) -u_{\epsilon}^\delta(t,i\delta) \right] +\rho(u_{\epsilon}^\delta(t,i\delta)) \ud M_{i\delta}^{\epsilon}(t)
, & t >0\,, i \in \bbZ^d\,,
\\[2em]
\displaystyle      
u_{\epsilon}^\delta (0,i\delta) = (\mu * G(\epsilon, \cdot))(i \delta)\,, & i \in \bbZ^d \,. \\
\end{cases}
\end{equation}
It has a strong solution 
\begin{align}
 \label{E:SSol2}
 \begin{aligned}
u_\epsilon^\delta(t,i\delta) = (\mu * G(\epsilon, \cdot))(i \delta)
+ &\int_0^t \frac{1}{\epsilon}\sum_{j\in\bbZ^d} P_{ij}^{\epsilon,\delta} \left[ u_{\epsilon}^\delta(s,j\delta) -u_{\epsilon}^\delta(s,i\delta) \right]\ud s\\
+& \int_0^t \rho(u_{\epsilon}^\delta(s,i\delta)) \ud M_{i\delta}^{\epsilon}(s).
\end{aligned}
\end{align}

We first note that \eqref{E:SSol2} is a discretization of \eqref{E:SSol1}: If we replace $x$ in \eqref{E:SSol1} by $[x]_\delta$ and set $i=[x/\delta]$, 
we see that the first and third terms on the r.h.s. of \eqref{E:SSol1} becomes the first and third terms on the r.h.s. of \eqref{E:SSol2}, respectively. The r.h.s. of \eqref{E_:SSol1} becomes
\begin{align*}  
\frac{1}{\epsilon}
\int_0^t\int_{\R^d} &G(\epsilon,i\delta-y)\left[u_\epsilon(s,y)- u_\epsilon(s,i\delta)\right]\ud y \ud s
\\
&\approx \frac{1}{\epsilon}
\int_0^t \ud s \sum_{j\in\bbZ^d}\int_{Q_\delta(j\delta)} G(\epsilon,i\delta-y)\left[u_\epsilon(s,j\delta)- u_\epsilon(s,i\delta)\right]\ud y \\
&= \frac{1}{\epsilon}
\int_0^t \ud s \sum_{j\in\bbZ^d}P_{ij}^{\epsilon,\delta}\left[u_\epsilon(s,j\delta)- u_\epsilon(s,i\delta)\right],
\end{align*}
which is equal to the second term on the r.h.s. of \eqref{E:SSol2} (then one may put a superscript $\delta$ in $u_\epsilon$ to denote the step size of this discretization). 

Note that $\{M_{i\delta}^\epsilon(t),t\ge 0\}_{i\in\bbZ}$ is a sequence of correlated Brownian motions starting from zero with 
\begin{align*}
\E\left(M_{i\delta}^\epsilon(t) M_{j\delta}^\epsilon(s)\right) = & 
(t\wedge s) \iint_{\R^{2d}}
G(\epsilon,i\delta-y_1)G(\epsilon,j\delta-y_2)f(y_1-y_2)\ud y_1\ud y_2\\
=& (t\wedge s) \int_{\R^d} \exp\left(-2\epsilon |\xi|^2\right) \cos(\delta(i-j)\cdot \xi)\; \widehat{f}(\ud\xi) ;
\end{align*}
see \eqref{E:d<M>t}.
% Or equivalently, 
% \[
% \int_0^t \frac{1}{\epsilon}\sum_{j\in\bbZ^d} P_{ij}^{\epsilon,\delta} \left[ u_{\epsilon}^\delta(s,j\delta) -u_{\epsilon}^\delta(s,i\delta) \right]\ud s
% =\frac{1}{\epsilon}
% \int_0^t\int_{\R^d} G(\epsilon,i\delta-y)\left[u_\epsilon^\delta(s,y)- u_\epsilon^\delta(s,i\delta)\right]\ud y \ud s
% \]

\bigskip
The main result of this step is the following lemma: 

\begin{lemma}\label{L:TwoStep}
Suppose that the initial data $\mu$ is bounded, i.e., $\mu(\ud x)=g(x)\ud x$ with $g\in L^\infty(\R^d)$. 
Let $u_\epsilon(t,x)$ be the strong solution \eqref{E:SSol1} to \eqref{E:apprx} and for any $\epsilon,\delta\in (0,1)$, let $u_\epsilon^\delta(t,[x]_\delta)$ be the system of stochastic differential equations given in \eqref{E:apprx2}. 
Then for any $t>0$, $x\in\R^d$ and $p\ge 2$,  it holds that 
\begin{align}\label{E:Step2} 
 \lim_{\delta\rightarrow 0_+}\sup_{x\in \R^d} \Norm{u_\epsilon(t,x) - u_\epsilon^\delta(t,[x]_\delta)}_p  =0. 
\end{align}
\end{lemma}
\begin{proof}
Fix arbitrary $p\ge 2$, $t>0$ and $x\in\R^d$. 
Notice that 
\begin{align*}
\Norm{u_\epsilon(t,x) - u_\epsilon^\delta(t,[x]_\delta)}_p 
\le  &\Norm{u_\epsilon(t,x)-u_\epsilon(t,[x]_\delta)}_p+ \Norm{u_\epsilon(t,[x]_\delta) - u_\epsilon^\delta(t,[x]_\delta)}_p\\
 =:& I_1^{\epsilon,\delta}(t,x)+I_2^{\epsilon,\delta}(t,x).
\end{align*}
For $I_1^{\epsilon,\delta}$, Lemma \ref{L:HolderXeps} shows that 
\begin{equation}\label{E:diff}
\sup_{x\in \R^d} I_1^{\epsilon,\delta}(t,x) \leq \sup_{s\in[0,t]}\sup_{j\in\bbZ^d}\sup_{x,y\in Q_\delta(j\delta)}
\Norm{u_\epsilon(s,x)-u_\epsilon(s,y)}_p\le C_\epsilon \delta.
\end{equation}
Now we study $I_2^{\epsilon,\delta}$. 
Denote $v_\epsilon^\delta(t, [x]_\delta):= u_\epsilon^\delta (t, [x]_\delta)-u_\epsilon(t, [x]_\delta)$. 
By setting $i=[x/\delta]$, we see that
\[
v_\epsilon(t, i\delta)= \sum_{\ell=1}^4  A_\ell^{\epsilon, \delta}(t, i\delta),
\]
where  
\begin{align*}
A_1^{\epsilon, \delta} (t, i\delta)&:= \frac{1}{\epsilon} \int_0^t\ud s \sum_{j\in \Z^d} \int_{Q_\delta(j\delta)}\ud y\:  G(\epsilon, i\delta-y) \left[ u_\epsilon(s, y)-u_\epsilon(s, j\delta)\right] ,\\
A_2^{\epsilon, \delta} (t, i\delta)&:=  \frac{1}{\epsilon} \int_0^t \ud s\sum_{j\in \Z^d} \int_{Q_\delta(j\delta)}\ud y\:  G(\epsilon, i\delta-y)v^\delta_\epsilon (s, j\delta), \\
A_3^{\epsilon, \delta} (t, i\delta)&:= -\frac{1}{\epsilon} \int_0^t  v^\delta_\epsilon (s, i\delta) \, \ud s,  \\
A_4^{\epsilon, \delta} (t, i\delta)&:= \int_0^t \left[
\rho\left( u_\epsilon^\delta (s, i\delta) \right) 
-\rho\left( u_\epsilon(s, i\delta) \right)
\right] dM^\epsilon_{i\delta}(s) .
\end{align*}
By Lemma \ref{L:HolderXeps} again (see also \eqref{E:diff}), \[
\Norm{A_1^{\epsilon, \delta} (t, i\delta)}_p \le C_\epsilon\delta.
\]
For $A_2^{\epsilon, \delta}$ and $A_3^{\epsilon, \delta}$,
by Minkowski's inequality, we see that 
\begin{align*}
\max_{\ell=1,2}\Norm{A_\ell^{\epsilon, \delta} (t, i\delta)}_p &
\le \frac{1}{\epsilon} \int_0^t  \sup_{j\in \Z^d}  \Norm{v_\epsilon^\delta(s, j\delta)}_p ds.
\end{align*}
For $A_4^{\epsilon, \delta}$, by the same argument as \eqref{E:I4BDG}, we see that  
\begin{align*} 
\Norm{A_4^{\epsilon, \delta}}_p^2 
 \le & C_\epsilon \int_0^t \sup_{j\in \Z^d}  \Norm{v_\epsilon^\delta(s, j\delta)}_p^2 \ud s.
\end{align*}
Combining these terms we see that 
\[
\sup_{j\in \Z^d}  \Norm{v_\epsilon^\delta(t, j\delta)}_p^2 
\le C_\epsilon\delta^2+ 
C_\epsilon \int_0^t \sup_{j\in \Z^d}  \Norm{v_\epsilon^\delta(s, j\delta)}_p^2 \ud s.
\]
An application of Gronwall's lemma implies that 
\begin{align}\label{E:I2}
I_2^{\epsilon,\delta} \le \sup_{j\in \Z^d}  \Norm{v_\epsilon^\delta(t, j\delta)}_p\le 
C_\epsilon \delta.
\end{align}
Finally, \eqref{E:diff} and  \eqref{E:I2} together prove \eqref{E:Step2}.
\end{proof}

\subsection{Step 4 (Proof of Theorems \ref{T:MomComp}, \ref{T2:MomComp} and Corollary \ref{C:SPDE:Slepian})}\label{SS:Step4}
We now combine Steps 1--3 above to  construct solutions to the infinite dimensional SDEs \eqref{E:apprx2} which converge in $L^p(\Omega)$ to the unique solution $u(t, x)$ to SHE \eqref{E:SHE} with rough initial data and driven by Gaussian noise whose spatial spectral measure only satisfies Dalang's condition \eqref{E:Dalang}. 
Let us fix arbitrary $t>0$, $x\in\R^d$ and $p\ge 2$.

Step 1 shows that  there exist a solution $u_{\epsilon_1, \epsilon_1'}(t, x)$ to \eqref{E:SHE} with bounded initial data (see Lemma \ref{L:initial}) and driven by Gaussian noise whose spatial spectral measure satisfies the strengthened Dalang's condition \eqref{E:DalangAlpha} with $\alpha=1$ such that 
\begin{align} 
\lim_{\epsilon_1'\to 0^+}\lim_{\epsilon_1\to 0^+} \|u(t, x)-u_{\epsilon_1, \epsilon_1'}(t, x) \|_p =0, \tag{by Lemma \ref{L:Approx}}
\end{align}
where $\epsilon_1$ (resp. $\epsilon_1'$) refers to the mollification for the noise (resp. the initial data) as in part (2) (resp. part (1)) of Lemma \ref{L:Approx}.

Step 2 now implies that there exists a strong solution $u_{\epsilon_1, \epsilon_1',\epsilon_2} (t,x)$ to \eqref{E:apprx} such that  
\begin{align}
\lim_{\epsilon_2\to 0^+} \Norm{u_{\epsilon_1, \epsilon_1'} (t, x)-u_{\epsilon_1, \epsilon_1',\epsilon_2}(t, x)}_p =0.  \tag{by Lemma \ref{L:ApproxLaplacian}}  
\end{align}

Step 3  shows that there exists a solution $u_{\epsilon_1, \epsilon_1',\epsilon_2}^{\delta}(t, [x]_\delta)$ to the infinite dimensional SDE \eqref{E:apprx2} such that 
\begin{align}
\lim_{\delta\to 0^+} \Norm{u_{\epsilon_1, \epsilon_1',\epsilon_2}(t, x)-u_{\epsilon_1, \epsilon_1',\epsilon_2}^{\delta}(t, [x]_\delta)}_p =0. \tag{by Lemma \ref{L:TwoStep}}
\end{align}

Now it is easy to check that \eqref{E:apprx2} is of the form 
\eqref{E:SDE1}, i.e., Assumption \ref{A:SDE} is satisfied. 
In particular, part (iv) of Assumption \ref{A:SDE} is satisfied thanks to Lemma \ref{L:initial}. Thus, an application of Theorems \ref{T:MomComSDE1} and  \ref{T:MomComSDE2} completes the proof of Theorems \ref{T:MomComp} and \ref{T2:MomComp}. 
Note that the two cases, namely, the multiple-time comparison over $\bbF[C_{p,+}^{2,v}]$ or $\bbF[C_{b,-}^{2,v}]$ and the single-time comparison over $\bbF[C_p^{2,v}]$, are treated separately in the proofs of Theorems \ref{T:MomComSDE1} and  \ref{T:MomComSDE2} below.
\bigskip

It remains to prove Corollary \ref{C:SPDE:Slepian}. Under condition (i), there exists some $\epsilon_0>0$ such that $f_1([-\epsilon_0,\epsilon_0]^d)=f_2([-\epsilon_0,\epsilon_0]^d)$, which, together with the fact that $f_1-f_2$ is a nonnegative measure, imply that for all $\epsilon\in (0,\epsilon_0]$, $f_1^{\epsilon, \epsilon}(0)=f_2^{\epsilon, \epsilon}(0)$ and $f_1^{\epsilon, \epsilon}(x)-f_2^{\epsilon, \epsilon}(x) \geq 0$ for all $x\in \R^d$ where $f^{\epsilon, \epsilon}_i$ is defined in Lemma \ref{L:Approx}. Thus, the result is a consequence of the approximation procedure and Corollary \ref{C:Slepian}. Under condition (ii), since $f_\ell \in C_b^2(\R^d;\R_+)$, $\ell=1,2$, $f_\ell$ have to satisfy properties in \eqref{E:fee}. Hence, in Step 1, we do not need to mollify the noise, or equivalently, we could set $\epsilon_1=0$.
We keep the approximations. Then one can apply Corollary \ref{C:Slepian} to conclude this case. This proves Corollary \ref{C:Slepian}. 
\qed

\section{Stochastic comparison principles for interacting diffusions} \label{S:MomComSDE}

In this section, we will study the interacting diffusion equations \eqref{E:SDE1} and prove Theorems \ref{T:ExtUniqSDE1}, \ref{T:MomComSDE1}, \ref{T:MomComSDE2}, \ref{T:FiniteCc}, and Corollary \ref{C:Slepian}.

\subsection{Existence and uniqueness (Proof of Theorem \ref{T:ExtUniqSDE1})}\label{SS:Existence}

\begin{proof}[Proof of Theorem \ref{T:ExtUniqSDE1}]
To show the existence of a solution, we use the standard Picard iteration. 
For $n=0$, set $U^{(0)} (t,i):=u_0(i)$ and for any $n\ge 1$, define recursively
\begin{align}\notag
 U^{(n+1)}(t,i)= & u_0(i) + \kappa \int_0^t \sum_{j\in K} p_{i,j} \left(U^{(n)}(s,j) - U^{(n)} (s, i) \right) \, \ud s  + \int_0^t \rho\left( U^{(n)}(s,i) \right)\ud M_i(s)\\
 =:& u_0(i) + I_n(t,i) + R_n(t,i).\label{E:Picard} 
\end{align}
Choose and fix an arbitrary integer $k\ge 2$.
Without loss of generality, we may assume $k$ is an even integer.
We first show that all $U^{(n)}(t,\cdot)$'s are in $\ell^k(K)$ almost surely for any $t\ge 0$. 
For any random field $Z(t,i)$, define
\[ 
\mathcal{N}_{\beta, k}(Z):=\sup_{t\geq 0} e^{-\beta t} 
\E \left(\|Z(t,\cdot)\|_{\ell^k(K)}^k\right),\qquad \beta\ge 0.
\] 
Note that $\mathcal{N}_{\beta, k}^{1/k}(Z)$ is a norm on the random field. 
Then by the Minkowski inequality, 
\begin{align}\notag
\calN_{\beta,k}\left(U^{(n+1)}\right)
&\le 
\left(
\calN_{\beta,k}^{1/k}(u_0) +\calN_{\beta,k}^{1/k}(I_n)
+\calN_{\beta,k}^{1/k}(R_n)
\right)^{k}\\
&\le 
3^{k-1} \left(\calN_{\beta,k}(u_0) 
+\calN_{\beta,k}(I_n)
+\calN_{\beta,k}(R_n)\right).
\label{E_:NU_n+1}
\end{align}
We will compute the three $\mathcal{N}_{\beta, k}^{1/k}(\cdot)$ norms in the right-hand side of \eqref{E_:NU_n+1}.
It is clear that 
\begin{align}\label{E_:Nu0}
\calN_{\beta,k}(u_0) = \Norm{u_0}_{\ell^k(K)}^k.
\end{align}
As for $\calN_{\beta,k}(I_n)$, because $k$ is even, we have that 
\begin{align*}
 &\sum_{i\in K} \E\left(\left| \int_0^t  \sum_{j\in K} p_{i,j} \left[ U^{(n)}(s,j) - U^{(n)}(s,i)\right] \ud s   \right|^k\right)\\
 = &
 \sum_{i\in K}\E
 \int_0^t\ud s_1\cdots\int_0^t\ud s_k
 \sum_{j_1\in K}\cdots \sum_{j_k\in K} 
 \prod_{\ell=1}^k  p_{i,j_\ell} \left[U^{(n)}(s_\ell,j_\ell)-U^{(n)}(s_\ell,i)\right]\\
  = &
 \sum_{i\in K}\E
 \int_0^t\ud s_1 e^{\beta s_1/k}\cdots\int_0^t\ud s_k e^{\beta s_k /k}
 \sum_{j_1\in K}\cdots \sum_{j_k\in K} 
 \left(\prod_{\ell=1}^k  p_{i,j_\ell} \right) \\
 &\quad \times \prod_{\ell'=1}^k
 e^{-\beta s_{\ell'}/k}\left[U^{(n)}(s_{\ell'},j_{\ell'})-U^{(n)}(s_{\ell'},i)\right].
 \end{align*}
By the inequality $\prod_{i=1}^k a_i\le (a_1^k+\cdots + a_k^k)/k$ applied to the product over $\ell'$, we see that
\begin{align*}
\prod_{\ell'=1}^k
 e^{-\beta s_{\ell'}/k}\left[U^{(n)}(s_{\ell'},j_{\ell'})-U^{(n)}(s_{\ell'},i)\right]
& \le 
 \frac{1}{k}\sum_{\ell'=1}^k
 \left(e^{-\beta s_{\ell'}}\left[U^{(n)}(s_{\ell'},j_{\ell'})-U^{(n)}(s_{\ell'},i)\right]^k\right)\\
 & \le 
 \frac{2^{k-1}}{k}\sum_{\ell'=1}^k
 \left(e^{-\beta s_{\ell'}}\left[U^{(n)}(s_{\ell'},j_{\ell'})^k+U^{(n)}(s_{\ell'},i)^k\right]\right).
 \end{align*}
Hence, 
\begin{align*}
 &\sum_{i\in K} \E\left(\left| \int_0^t  \sum_{j\in K} p_{i,j} \left[ U^{(n)}(s,j) - U^{(n)}(s,i)\right] \ud s   \right|^k\right)\\
 \le &
 \frac{2^{k-1}}{k}
 \sum_{\ell'=1}^k
 \int_0^t\ud s_1 e^{\beta s_1/k}\cdots\int_0^t\ud s_k e^{\beta s_k/k}
 \sum_{i\in K}\sum_{j_1\in K}\cdots \sum_{j_k\in K} 
 \left(\prod_{\ell=1}^k  p_{i,j_\ell}\right) \\
 &\quad \times \left[e^{-\beta s_{\ell'}} \E \left(U^{(n)}(s_{\ell'},j_{\ell'})^k\right)+e^{-\beta s_{\ell'}}\E\left( U^{(n)}(s_{\ell'},i)^k\right)\right]\\
 \le& \frac{2^{k-1}}{k}(1+\Lambda)
 \sum_{\ell'=1}^k
 \int_0^t\ud s_1 e^{\beta s_1/k}\cdots\int_0^t\ud s_k e^{\beta s_k/k}
 \left[e^{-\beta s_{\ell'}} \E\left(\Norm{U^{(n)}(s_{\ell'},\cdot)}_{\ell^k(K)}^k\right)\right]\\
 =& \frac{2^{k-1}}{k} (1+\Lambda)\: \calN_{\beta,k}\left(U^{(n)}\right)
 \sum_{\ell'=1}^k
 \int_0^t\ud s_1 e^{\beta s_1/k}\cdots\int_0^t\ud s_k e^{\beta s_k/k}
 \\
 \leq & 2^{k-1}(1+\Lambda)\: \calN_{\beta,k}\left(U^{(n)}\right)
 \left(\frac{k}{\beta}\right)^k e^{t\beta},
\end{align*}
where we have used the assumption \eqref{E:Lambda} and the fact that $\sum_{j\in K}p_{i,j}=1$.
Therefore, 
\begin{align}
 \label{E_:NIn}
\calN_{\beta,k}(I_n)\le 
\frac{1+\Lambda}{2}\left(\frac{2\kappa k}{\beta}\right)^k\calN_{\beta,k}\left(U^{(n)}\right).
\end{align}
 
Now let us consider $\calN_{\beta,k}(R_n)$. By the Burkholder-Davis-Gundy inequality, we have that
\begin{align*}
\E\left(\left| \int_0^t \rho\left( U^{(n)}(s,i)\right) \ud M_i(s)\right|^k\right)
\leq&
c_k \E\left(\left[ \int_0^t \rho\left( U^{(n)}(s,i)\right)^2 \gamma(0) \ud s\right]^{k/2}\right)\\
\le& 
c_k \Lip_\rho^k \gamma(0)^{k/2}
\int_0^t\ud s_1\cdots 
\int_0^t\ud s_{k/2}\E\left(\prod_{\ell=1}^{k/2}U^{(n)}(s_\ell,i)^2\right),
\end{align*}
where $c_k$ is some universal constant and we have used the fact that $k$ is an even integer. 
By the same arguments as above, 
\begin{align*}
&\int_0^t\ud s_1\cdots 
\int_0^t\ud s_{k/2}\E\left(\prod_{\ell=1}^{k/2}U^{(n)}(s_\ell,i)^2\right)
\\
=&  \int_0^t\ud s_1 e^{2\beta s_1/k}\cdots 
\int_0^t\ud s_{k/2}e^{2\beta s_{k/2}/k}
\E\left(\prod_{\ell=1}^{k/2}e^{-2\beta s_\ell/k} U^{(n)}(s_\ell,i)^2\right)\\
\le & \frac{2}{k}\int_0^t\ud s_1 e^{2\beta s_1/k}\cdots 
\int_0^t\ud s_{k/2}e^{2\beta s_{k/2}/k} \sum_{\ell=1}^{k/2}
e^{-\beta s_\ell}   \E\left(U^{(n)}(s_\ell,i)^k\right). 
\end{align*}
Thus, 
\begin{align*}
&\sum_{i\in K}\E\left| \int_0^t \rho\left( U^{(n)}(s,i)\right) \ud M_i(s)\right|^k\\
&\le 
c_k \Lip_\rho^k \gamma(0)^{k/2}\frac{2}{k}\int_0^t\ud s_1 e^{2\beta s_1/k}\cdots 
\int_0^t\ud s_{k/2}e^{2\beta s_{k/2}/k} \sum_{\ell=1}^{k/2}
e^{-\beta s_\ell} \:  \E\left(\Norm{U^{(n)}(s_\ell,\cdot)}_{\ell^k(K)}^k \right)\\
&\le 
c_k \Lip_\rho^k \gamma(0)^{k/2}\calN_{\beta,k}\left(U^{(n)}\right) \int_0^t\ud s_1 e^{2\beta s_1/k}\cdots 
\int_0^t\ud s_{k/2}e^{2\beta s_{k/2}/k} \\
&= 
c_k \Lip_\rho^k \gamma(0)^{k/2}\calN_{\beta,k}\left(U^{(n)}\right)\left(\frac{k}{2\beta}\right)^{k/2} e^{\beta  t}, 
\end{align*}
which implies that 
\begin{align}\label{E_:NRn}
\calN_{\beta,k}(R_n)
\le 
c_k \Lip_\rho^k \gamma(0)^{k/2}\left(\frac{k}{2\beta}\right)^{k/2} \calN_{\beta,k}\left(U^{(n)}\right).
\end{align}

Putting \eqref{E_:Nu0}, \eqref{E_:NIn} and \eqref{E_:NRn} back to \eqref{E_:NU_n+1} shows that
\[
\calN_{\beta,k}\left(U^{(n+1)}\right)
\le 3^{k-1} 
\Norm{u_0}_{\ell^k(K)}^k 
+C_k(\beta)\: \calN_{\beta,k}\left(U^{(n)}\right),
\]
where
\begin{align}\label{E:CkBeta}
C_k(\beta):=3^{k-1} \left[\frac{1+\Lambda}{2}\left(\frac{2\kappa k}{\beta}\right)^k+
c_k \Lip_\rho^k \gamma(0)^{k/2}\left(\frac{k}{2\beta}\right)^{k/2} \right].
\end{align}
It is clear that $\beta\mapsto C_k(\beta)$ is a strictly decreasing function  for all $\beta\ge 0$. 
Therefore, by choosing $\beta_*$ to be the unique positive solution to the equation $C_k(\beta)=1/2$, we have that
\begin{align}\notag
\calN_{\beta_*,k}\left(U^{(n+1)}\right)
\le & 
3^{k-1} \Norm{u_0}_{\ell^k(K)}^k 
+\frac{1}{2} \calN_{\beta_*,k}\left(U^{(n)}\right)\\ \notag
\le & 
3^{k-1} \Norm{u_0}_{\ell^k(K)}^k \left(1+\frac{1}{2}+\cdots\right)\\
\le & 3^{k} 
\Norm{u_0}_{\ell^k(K)}^k.
\label{E_:Nun+1}
\end{align}
Therefore, we have that
\begin{equation} \label{E:MomBound}
\sup_{n\geq 0} \sup_{0\leq t\leq T} \E \left(\Norm{U^{(n)}(t,\cdot)}^k_{\ell^k(K)}\right) \leq  3^k \|u_0\|^k_{\ell^k(K)} e^{\beta_* T}. 
\end{equation}
This implies that $U^{(n)}(t,\cdot)$ for all $n\geq 1$ is well-defined and in $L^\infty\left([0,T]; L^k(\Omega; \ell^k(K))\right)$.  

Let $V^{(n)}(t,i):=U^{(n+1)}(t,i)-U^{(n)}(t,i)$ for $n\ge 0$. 
Since $\rho$ is globally Lipschitz with Lipschitz constant $\Lip_\rho$, following the same process as above, we can have 
\[
  \mathcal{N}_{\beta_*,k}\left( V^{(n)}\right) \leq  \frac{1}{2}\mathcal{N}_{\beta_*,k}\left( V^{(n-1)}\right)\leq \cdots \leq \frac{1}{2^n} \mathcal{N}_{\beta_*,k}\left( V^{(0)}\right).
\]
Since $V^{(0)} = U^{(1)} -u_0$, we see that
\begin{align*}
\calN_{\beta_*,k}\left(V^{(0)}\right) \le &
\left[\calN_{\beta_*,k}^{1/k}\left(U^{(1)}\right)+\calN_{\beta_*,k}^{1/k}\left(u_0\right)\right]^k\\
\le & 2^{k-1} \calN_{\beta_*,k}\left(U^{(1)}\right)+2^{k-1} \Norm{u_0}_{\ell^k(K)}^k\\
\le & 2^{k-1}\left(3^k+1\right) \Norm{u_0}_{\ell^k(K)}^k.
\end{align*}
Thus, $\sum_{n=0}^\infty \mathcal{N}_{\beta_*,k}^{1/k}\left( V^{(n)}\right)<\infty$, which implies that $\left\{U^{(n)}\right\}_{n\in\bbN}$ is a Cauchy sequence in the Banach space with the norm $\calN_{\beta_*,k}^{1/k}(\cdot)$.
As a consequence,
\[
u:=\lim_{n\to \infty} U^{(n)} \quad\text{ in $L^\infty\left([0,T]; L^k(\Omega; \ell^k(K))\right)$}.  
\]
Fatou's lemma and \eqref{E:MomBound} imply that 
\begin{equation} 
\sup_{0\leq t\leq T} \E \left(\Norm{U(t,\cdot)}^k_{\ell^k(K)}\right) \leq 3^k\|u_0\|^k_{\ell^k(K)} \exp\left(\beta_* T \right),
\end{equation} 
which, together with Lemma \ref{L:k^2} below, proves the second inequality in \eqref{E:lkMom}.
The first inequality in \eqref{E:lkMom} is due to the trivial fact that 
\begin{align}\label{E:Sup<Lp}
\sup_{i\in K}|U(t,i)|^p \le \sum_{i\in K}|U(t,i)|^p,\qquad \text{a.s.} 
\end{align}
In addition, using the convergence from $U^{(n)}$ to $U$ in $L^\infty\left([0,T]; L^2(\Omega; \ell^2(K))\right)$, it is easy to see that $U$ satisfies \eqref{E:SDE1-Strong}. 
The proof of uniqueness follows from a standard argument. We will not repeat here.
%\begin{equation}\label{E:SDEinteg}
%U(t,i) =u_0(i)+ \kappa \int_0^t \sum_{j\in K} p_{i,j} \left( -\delta(i,j))u(s,j) \, \ud s + \int_0^t\rho (u(s,i)) \ud M_i(s), \quad i\in K, t>0.
%\end{equation}
% 
% As for \eqref{E:lkMom}, one needs to work under the following norm, for  a random field $Z(t,i)$,
% \[ 
% \mathcal{M}_{\beta,k}(Z):=\sup_{t\geq 0} \sup_{i\in K} e^{-\beta t}\, \E |Z(t,i)|^k. 
% \]
% %Recall $\|X\|_k:=\left(\E |X|^k\right)^{1/k}$. 
% Then by similar arguments as above, one can show that
% \begin{equation*}
% \mathcal{M}_{\beta, k} \left( U^{(n+1)} \right) \leq 4^k \sup_{i\in K} |u_0(i)|^k +   \frac{1}{2}\mathcal{M}_{\beta, k} \left( U^{(n)} \right),
% \end{equation*}where $\beta:=\max\{32k\kappa, 64k\Lip_\rho^2 \gamma(0)\}$.  Since $u_0\in \ell^2(K) \subset \ell^\infty(K)$, Fatou's lemma implies \eqref{E:lkMom}. 
This completes the proof of Theorem \ref{T:ExtUniqSDE1}.
\end{proof}

% \begin{remark}
% If one wants to study intermittency property of the solution, one needs to know the dependence of $\beta_*$ on $k$. Indeed, it is known that the constant $c_k$ in the Burkholder-Davis-Gundy inequality can be chosen to be $c_k = 2^k k^{k/2}$ (see, e.g., \cite{ChenDalang13Heat}).  Then by solving the equation $C_k(\beta)=1/2$, which is solvable explicitly since it is an essentially a quadratic equation, one finds that the growth of $\beta_*$ with resepect to $k$ is indeed $k^2$.
% \end{remark}

\begin{lemma}\label{L:k^2}
Let $C_k(\beta)$ be the constant defined in \eqref{E:CkBeta} and $\beta_*$ be the positive solution to the equation $C_k(\beta)=1/2$. Then for some constant $C=C(\kappa,\gamma(0),\Lip_\rho,\Lambda)$, it holds that $\beta_*\le C k^2$ for all $k\ge 2$.
\end{lemma}
\begin{proof}
For simplicity, we will only prove the symmetric case, i.e., $p_{i,j}=p_{j,i}$, in which case, $\Lambda\equiv 1$. 
We first remark that the constant $c_k$ in the Burkholder-Davis-Gundy inequality can be chosen to be $c_k = 2^k k^{k/2}$ (see, e.g., \cite{ChenDalang13Heat}). In order to solve the equation 
$C_k(\beta)=1/2$, set $x=\beta^{-k/2}$. Then equivalently, $x$ solves 
\[
\left(6\kappa k\right)^k x^2+
\left(18\gamma(0) k^2\Lip_\rho^2\right)^{k/2} x
-3/2=0.
\]
By finding the positive solution and then taking the power of $-2/k$, we see that
\begin{align*}
\beta_* =&
2 \times  3^{1-2/k}
   \left( \left[3\gamma k^2 \Lip_\rho^2\right]^{k/2}+\sqrt{6
   (\kappa  k)^k+\left[3\gamma k^2 \Lip_\rho^2\right]^k}\right)^{2/k}\\
\le & 2 \times  3^{1-2/k}
   \left( 2\left[3\gamma k^2 \Lip_\rho^2\right]^{k/2}+\sqrt{6}
   (\kappa  k)^{k/2}\right)^{2/k}\\
\le &   
2 \times  3^{1-2/k}
   \left( 2^{2/k}3\gamma k^2 \Lip_\rho^2 +6^{1/k}
   \kappa  k\right)    \\
   \le &   
6\left( 3\gamma k^2 \Lip_\rho^2 +\kappa  k\right),    
\end{align*}
where we have applied twice the subadditivity property of the function $\R_+ \ni x\mapsto x^{\alpha}$ for $\alpha\in (0,1]$.
Therefore, one can find a constant $C$ depending on $\kappa$, $\gamma(0)$ and $\Lip_\rho$ such that $\beta_*\le C k^2$ for all $k\ge 2$.
\end{proof}

\subsection{Several approximations}\label{SS:Aprox2}

In this subsection, we will reduce the SDE \eqref{E:SDE1} to the case in Theorem \ref{T:FiniteCc}. We will need to approximate the solution to \eqref{E:SDE1} in the following two cases:
The first case is to approximate \eqref{E:SDE1} by that with a $C_c^2(\R_+)$ diffusion coefficient. This is covered in two steps through Propositions \ref{P:rhoBD} and \ref{P:rho2} below.
The second case is to approximate \eqref{E:SDE1} by a finite dimensional SDE and this is covered by Proposition \ref{P:Approx1} below. 

\bigskip

{\bf\noindent Case 1. ~}
Define 
\begin{align}
\rho_N(x):=\rho(x)\one_{\{|x|\le N\}} + \rho\left(\sgn(x) N\right)\left(2-|x|/N\right)\one_{\{N\le |x|\le 2N\}}.
\end{align}
Since $\rho$ is a globally Lipschitz continuous function with the Lipschitz constant $\Lip_\rho$ and $\rho(0)=0$, it is easy to see that $\rho_N$ is also globally Lipschitz such that
\begin{align}\label{E:LipNLip}
\Lip_{\rho_N}\le \Lip_{\rho} \qquad \text{and}\qquad \rho_N(0)=0, 
\end{align}
which imply that $|\rho_N(x)| \leq \Lip_\rho|x|$. Consider

\begin{equation}\label{E:SDE3}
 \left\{
\begin{array}{ll}
\displaystyle  
 \ud U_N(t,i)=\kappa \sum_{j\in K} p_{i,j}\left( U_N(t, j)- U_N(t, i)\right)  \, \ud t + \rho_N (U_N(t,i)) \ud M_i(t),
&i\in K, t>0, 
\\[2em]
\displaystyle      
U_N (0,i) =u_0(i)\,, & i \in K.\,\\
\end{array}
\right.
\end{equation}
The existence and uniqueness of a strong solution in the space \eqref{E:DiscreteSpace} follows from Theorem \ref{T:ExtUniqSDE1}.

\begin{proposition}\label{P:rhoBD}
Let $U(t,i)$ and $U_N(t,i)$ be solutions to \eqref{E:SDE1} and \eqref{E:SDE3}, respectively, with the same initial data $u_0(\cdot)\in \ell^2(K)$.  Then, for any $T> 0$ and $k\ge 2$, we have that
\begin{equation}\label{E:MomBDrho1}
\sup_{t\in[0,T]}\E\left(\sup_{i\in K} \left|U(t,i)-U_N(t,i)\right|^k \right)
\le 
\sup_{t\in[0,T]}\E\left( \Norm{U(t,i)-U_N(t,i)}_{\ell^k(K)}^k \right)
\rightarrow 0,
\end{equation}
as $N\rightarrow+\infty$.
\end{proposition} 

\begin{proof}
Let $T\ge t\ge 0$ and fix $k\ge 2$. Without loss of generality, we may assume that $k$ is an even integer. 
Let $V_N(t,i):=U(t,i)-U_N(t,i)$. Then, $V_N(t, i)$ is a solution to
\begin{align*} 
\ud  V_N (t, i)&= \kappa \sum_{j\in K} p_{i, j} V_N(t,j)\ud t -\kappa V_N(t,i)\ud t + \left( \rho_N(U(t,i)) -\rho_N(U_N(t, i)) \right) \ud M_i(t)\\
& \quad + \left( \rho(U(t,i)) -\rho_N(U(t, i)) \right) \ud M_i(t). 
\end{align*}
By It\^o's formula 
\begin{align*} 
\ud  V_N^k(t, i)=& \kappa k  V_N^{k-1}(t, i) \sum_{j\in K} p_{i, j} V_N(t,j)\ud t -\kappa k V_N^k(t,i)\ud t \\
&+ k  V_N^{k-1}(t, i)\left( \rho_N(U(t,i)) -\rho_N(U_N(t, i)) \right) \ud M_i(t)\\
& + k  V_N^{k-1}(t, i)\left( \rho(U(t,i)) -\rho_N(U(t, i)) \right) \ud M_i(t)\\
&+ \frac{k(k-1)}{2}\gamma(0)\,  V_N^{k-2}(t, i)\left( \rho_N(U(t,i)) -\rho_N(U_N(t, i)) \right)^2 \ud t\\
& + \frac{k(k-1)}{2}\gamma(0) \, V_N^{k-2}(t, i)\left( \rho(U(t,i)) -\rho_N(U(t, i)) \right)^2 \ud t. 
\end{align*}
By the following Young's inequality for product
\begin{align}\label{E:Young}
k a b^{k-1}\le a^k+(k-1)b^k,\qquad \text{for all $a,b\ge 0$ and $k\ge 2$,}
\end{align}
we see that 
\begin{align*}
\sum_{i\in K}\left|\kappa k  V_N^{k-1}(t, i) \sum_{j\in K} p_{i, j} V_N(t,j) \right|
&\le \kappa \sum_{i\in K}\sum_{j\in K}p_{i,j}\left(
V_N^k(t,j) + (k-1)V_N^k(t,i)\right)\\
&\le \kappa (k-1+\Lambda) \Norm{V_N(t,\cdot)}_{\ell^k(K)}^k,
\end{align*}
where we have used the assumption \eqref{E:Lambda} and the fact that $\sum_{j\in K}p_{i,j}=1$. 

By \eqref{E:LipNLip},
\begin{align*}
\sum_{i\in K}
V_N^{k-2}(t, i)\left( \rho_N(U(t,i)) -\rho_N(U_N(t, i)) \right)^2 
\le \Lip_\rho^2 \Norm{V_N(t,\cdot)}_{\ell^k(K)}^k.
\end{align*}
 By Young's inequality \eqref{E:Young} with $k/2$, we see that
\begin{align*}
\frac{k}{2}
\left( \rho(U(t,i)) -\rho_N(U(t, i)) \right)^2V_N^{k-2}(t,i)&\le
\left( \rho(U(t,i)) -\rho_N(U(t, i)) \right)^k
+ \frac{k-2}{2}V_N^k(t,i).
\end{align*}
Hence,
\begin{align*}
&\sum_{i\in K}\frac{k(k-1)}{2}\gamma(0) \, V_N^{k-2}(t, i)\left( \rho(U(t,i)) -\rho_N(U(t, i)) \right)^2 \\
\le  &(k-1)\gamma(0) \Norm{\rho(U(t,\cdot)) -\rho_N(U(t, \cdot))}_{\ell^k(K)}^k+ \frac{(k-1)(k-2)}{2}\gamma(0) \Norm{V_N^k(t,\cdot)}_{\ell^k(K)}^k.
\end{align*}

Therefore,
\begin{align*}
\E\left(\Norm{V_N(t,\cdot)}_{\ell^k(K)}^k\right) 
\le & \quad C_1 \int_0^t \E\left(\Norm{V_N(s,\cdot)}_{\ell^k(K)}^k\right) \ud s\\
&+C_2 \int_0^t \E\left( \Norm{\rho(U(s,\cdot)) -\rho_N(U(s, \cdot))}_{\ell^k(K)}^k\right)\ud s,
\end{align*}
where the two constants can be chosen as follows: 
\[
C_1:= 2 \kappa (k-1+\Lambda)+\frac{k(k-1)}{2}\gamma(0)\Lip_\rho^2+ \frac{(k-1)(k-2)}{2}\gamma(0)
\quad\text{and}\quad 
C_2:=(k-1)\gamma(0).
\]
By setting $W_N(t):=\sup_{s\in [0,t]}\E\left(\Norm{V_N(s,\cdot)}_{\ell^k(K)}^k\right)$, we see that
\begin{align}\label{E:WN}
W_N(t) \le C_1 \int_0^t W_N(s)\ud s + 
C_2 \int_0^t \sup_{s'\in [0,s]}\E\left( \Norm{\rho(U(s',\cdot)) -\rho_N(U(s', \cdot))}_{\ell^k(K)}^k\right)\ud s.
\end{align}
Because $|\rho_N(x)|\leq \Lip_\rho |x|$ for all $x\in \R$,  the moment bound \eqref{E:lkMom} implies that 
\begin{align*}
&\sup_{0\leq s\leq t}   \E\left(
\Norm{\rho(U(s,\cdot)) - \rho_N(U(s, \cdot))}_{\ell^k(K)}^k\right) \\
 \leq &\sup_{0\leq s\leq t} \E\left(\sum_{i\in K}\left|\rho(U(s,i)) - \rho_N(U(s, i)) \right|^k \one_{\{|U(s,i)|\ge N\}} \right)\\
\leq &2^{k-1}\Lip_\rho^k \sup_{0\leq s\leq t}   \E\left(\sum_{i\in K} |U(s,i)|^k \mathbbm{1}_{\{|U(s,i)|>N\}} \right)
\rightarrow \: 0,\qquad\text{as $N\rightarrow+\infty$.} 
\end{align*}
On the other hand, the above inequality shows that
\begin{align*}
\sup_{0\leq s\leq t}   \E\left(\sum_{i\in K}\left|\rho(U(s,i)) - \rho_N(U(s, i)) \right|^k \right) 
 \le &2^{k-1}\Lip_\rho^k \sup_{0\leq s\leq t}   \E\left(\Norm{U(s,\cdot)}_{\ell^k(K)}^k \right)\le C_T,
\end{align*}
for all $t\in [0,T]$. Hence, the second term on the right-hand side of \eqref{E:WN} converges to zero as $N\rightarrow\infty$ by the dominated convergence theorem. Moreover, as a function of $t$, this term is in $L^1([0,T])$. Therefore, an application of Gronwall's lemma completes the proof. 
\end{proof}

%\begin{remark}
%Thanks to Theorem \ref{T:rhoBD} and \eqref{E:MomBDrho}, we have  
Thanks to Proposition \ref{P:rhoBD}, we can now assume that $\rho$ is a function with compact support. Let $\phi \in C_c^\infty(\R)$ with $\int_{\R} \phi(x) \ud x =1$ and let $\phi_\epsilon(x)=\epsilon^{-1} \phi(x/\epsilon)$. Define $\rho_\epsilon(x):=\phi_\epsilon * \rho (x)$. Then, it is easy to see that $\rho_\epsilon \in C_c^\infty$ and $\rho_\epsilon$ is globally Lipschitz with the same Lipschitz constant $\Lip_\rho$ as for $\rho$. Consider 

\begin{equation}\label{E:SDE4}
 \left\{
\begin{array}{ll}
\displaystyle  \ud U_\epsilon(t,i)=\kappa \sum_{j\in K} p_{i,j}\left( U_\epsilon(t, j) - U_\epsilon(t,i) \right) \ud t + \rho_\epsilon (u_\epsilon(t,i)) \ud M_i(t), \quad &i\in K, t>0, 
\\[1em]
\displaystyle      
U_\epsilon (0,i) =u_0(i)\,, & i \in K.\,\\
\end{array}
\right.
\end{equation}
The existence and uniqueness of a strong solution in the space \eqref{E:DiscreteSpace} comes from Theorem \ref{T:ExtUniqSDE1}.

\begin{proposition}\label{P:rho2}
Let $U(t,i)$ and $U_\epsilon(t,i)$ be solutions to \eqref{E:SDE1} and \eqref{E:SDE4}, respectively, with the same initial data $u_0(\cdot)\in\ell^2(K)$ and with $\rho$ being a continuous function with compact support. Then, for any $T> 0$ and $k\ge 2$, it holds that
\begin{equation}\label{E:MomBDrho2}
 \sup_{t\in [0,T]}\E\left(\sup_{i\in K}\left|U(t,i)-U_\epsilon(t,i)\right|^k \right) 
 \le 
 \sup_{t\in [0,T]}\E\left(\Norm{U(t,\cdot)-U_\epsilon(t,\cdot)}_{\ell^k(K)}^k \right)\rightarrow 0,
\end{equation}
as $\epsilon\rightarrow0_+$.
\end{proposition} 
\begin{proof}
Since $\rho$ is a continuous function with compact support, we have that $\Norm{\rho_\epsilon}_{L^\infty(\R)}\le \Norm{\rho}_{L^\infty(\R)} <\infty$. 
On the other hand,  since both $\rho$ and $\rho_\epsilon$ are continuous functions with compact support, $\rho_\epsilon$ converges to $\rho$ uniformly on any compact set, i.e., 
\[
\lim_{\epsilon\to 0} \sup_{0\leq s\leq t} \sup_{i\in K} \left|\rho(U(s,i)) -\rho_\epsilon(U(s, i)) \right|^k=0.
\]
Hence, the bounded convergence theorem implies that 
\[
\sup_{0\leq s\leq t} \E\left(\sup_{i\in K} \left|\rho(U(s,i)) -\rho_\epsilon(U(s, i)) \right|^k\right)\le 
\E\left(\sup_{0\leq s\leq t} \sup_{i\in K} \left|\rho(U(s,i)) -\rho_\epsilon(U(s, i)) \right|^k\right) \rightarrow 0,
\]
as $\epsilon\rightarrow 0_+$.
Therefore, one can follow the same arguments as those in proposition \ref{P:rhoBD} to complete the proof.   
\end{proof}

{\bigskip\bf\noindent Case 2.~}
It remains to show the approximation by a finite-dimensional SDE when $K$ has countably infinite many elements. Let $K_i$ be subsets of $K$ with finite cardinalities such that $K_1 \subset K_2\subset \cdots \uparrow K$. Consider  the following finite system of interacting diffusions:

\begin{equation}\label{E:SDE2}
 \left\{
\begin{array}{ll}
\displaystyle  
\begin{aligned}
 \ud U_m(t,i)=&\kappa \sum_{j\in K_m} p_{i, j} \ U_m(t, j) \, \ud t -\kappa U_m(t,i) \, \ud t\\
& \hspace{6em}+ \rho \left(U_m(t,i))\right) \ud M_i(t),
\end{aligned}
\qquad &i\in K_m,\,\, t>0, 
\\[1em]
\displaystyle      
U_m (0,i) =u_0(i)\,, & i \in K_m, \,
\\[1em]
\displaystyle
U_m(t,i)=u_0(i)\,, & i\in K\setminus K_m,\,\, t\geq 0.
\end{array}
\right.
\end{equation}

The existence and uniqueness of a strong solution to \eqref{E:SDE2} is a standard result. Indeed, one may also follow the proof of Theorem \ref{T:ExtUniqSDE1} to show the existence of a unique strong solution.

\begin{proposition}\label{P:Approx1}
Let $U(t,i)$ and $U_m(t,i)$ be solutions to \eqref{E:SDE1} and \eqref{E:SDE2}, respectively, with the same diffusion coefficient $\rho$ which is assumed to be globally Lipschitz continuous. Then, for any $T> 0$ and $k\ge 2$, it holds that
\begin{equation}\label{E:UmU1}
\sup_{t\in [0,T]}\E\left(\sup_{i\in K}\left|U_m(t,i)-U(t,i)\right|^k\right)\le 
\sup_{t\in [0,T]}\E\left(\left\| U_m(t,\cdot) - U(t,\cdot) \right\|^k_{\ell^k(K)}\right) 
\rightarrow 0,
\end{equation}
as $m\rightarrow+\infty$.
\end{proposition} 

\begin{proof}
Let $T\ge t\ge 0$ and fix $k\ge 2$. Without loss of generality, we assume that $k$ is an even integer. Set $V_m(t,i):=U(t,i)-U_m(t,i)$. Then, $V_m(t,i)$ solves the following SDE
\begin{align*}
\ud V_m(t,i) =& 
\begin{cases}
\displaystyle
\begin{aligned}
&\kappa \sum_{j\in K_m} p_{i,j} V_m(t,j) \ud t  -\kappa V_m(t,i)\ud t \\
&\qquad  + \left(\rho(U(t,i))-\rho\left( U_m(t,i)\right) \right) \ud M_i(t)+ \sum_{j\in K\setminus K_m} p_{i,j}U(t,j)\ud t,
\end{aligned} & \text{if $i\in K_m$,} \\[1em]
\displaystyle \ud U(t,i) = \text{the r.h.s. of the first equation in \eqref{E:SDE1},} & \text{otherwise.}
\end{cases} 
\end{align*}
By It\^o's formula, we see that, for any $i\in K_m$,
\begin{align*}
\ud  V_m^k(t,i)=& k\kappa  \sum_{j\in K_m} p_{i,j}  V_m^{k-1}(t,i)V_m(t,j) \ud t - k\kappa  V_m^k(t,i) \ud t \\
&+k V_m^{k-1}(t,i)\left(\rho(U(t,i))-\rho\left( U_m(t,i)\right) \right) \ud M_i(t)\\
&+ k V_m^{k-1}(t,i) \sum_{j\in K\setminus K_m} p_{i,j}U(t,j)\ud t \\
& + \frac{k(k-1)}{2}V_m^{k-2}(t,i)\gamma(0)\left(\rho(U(t,i))-\rho\left( U_m(t,i)\right) \right)^2  \ud t.
\end{align*}
Notice that 
\[
\Norm{V_m(t,\cdot)}_{\ell^k(K)}^k = 
\Norm{V_m(t,\cdot)}_{\ell^k(K_m)}^k + \Norm{V_m(t,\cdot)}_{\ell^k(K\setminus K_m)}^k.
\] 
It is clear that 
\begin{align*}
\Norm{V_m(t,\cdot)}_{\ell^k(K\setminus K_m)}
&= 
\Norm{U(t,\cdot)-u_0(\cdot)}_{\ell^k(K\setminus K_m)}
 \le \Norm{U(t,\cdot)}_{\ell^k(K\setminus K_m)}
+\Norm{u_0(\cdot)}_{\ell^k(K\setminus K_m)}. 
\end{align*}
Theorem \ref{T:ExtUniqSDE1} says that $U(t, \cdot) \in \ell^2(K)\subseteq \ell^k(K)$ a.s. for all $t\geq 0$, which implies that   
\[
\sum_{j\in K\setminus K_m} U(t, j)^k \to 0 \quad\text{ as $m\to \infty$ a.s.} 
\]
Therefore, thanks to \eqref{E:lkMom}, the monotone convergence theorem implies that 
\begin{equation}\label{E:Uk}
\lim_{m\to\infty} \int_0^t 
\sup_{s'\in [0,s]}\E\left(\sum_{j\in K\setminus K_m} U(s', j)^k \right) \ud s=0.
\end{equation} 
In addition, since $u_0(\cdot)\in \ell^2(K)\subseteq \ell^k(K)$, we can get
\begin{align}\label{E_:VmKKm}
\lim_{m\rightarrow\infty}
\int_0^t \sup_{s'\in [0,s]}
\E \left(\Norm{V_m(s',\cdot)}_{\ell^k(K\setminus K_m)}^k \right)
\ud s =0.
\end{align}

As for $\Norm{V_m(t,\cdot)}_{\ell^k(K_m)}^k$,  
by Young's inequality \eqref{E:Young}, we have that
\begin{align*}
k\sum_{i\in  K_m } \sum_{j\in K_m} p_{i,j} 
\left|V_m^{k-1}(t,i)V_m(t,j) \right| 
&\leq \sum_{i\in K} \sum_{j\in K} p_{i, j} V_m^k(t,j)+ (k-1)\sum_{i\in K} \sum_{j\in K} p_{i,j}
V_m^k(t,i)   \\
&\le (k-1+\Lambda) \Norm{V_m(t,\cdot)}_{\ell^k(K)}^k,
\end{align*}
where we have used the assumption \eqref{E:Lambda} and the fact that $\sum_{j\in K}p_{i,j}=1$.
Similarly, by \eqref{E:Young},
\begin{align*}
\sum_{i\in   K_m } \left|k V_m^{k-1}(t,i) \sum_{j\in K\setminus K_m} p_{i, j} U(t,j)\right|& \leq  (k-1)\left\|  V_m(t,\cdot)  \right\|^k_{\ell^k(K)}+ \sum_{i \in K} \left[ \sum_{j\in K\setminus K_m} p_{i, j} U(t,j)\right]^k,
\end{align*}
where, by H\"older inequality and the fact that $\sum_{j\in K\setminus K_m}p_{i,j}\le 1$, 
\begin{align*}
\sum_{i \in K} \left[ \sum_{j\in K\setminus K_m} p_{i, j} U(t,j)\right]^k &\le \sum_{i \in K} \sum_{j\in K\setminus K_m} p_{i, j} U(t,j)^k\leq \Lambda \sum_{j\in K\setminus K_m} U(t, j)^k.
\end{align*}
Now combine things together and use the fact that $\rho$ is globally Lipschitz to see that
\begin{align*}
\E \left(\left\|  V_m(t,\cdot)  \right\|^k_{\ell^k(K)}\right) =&
\E \left( \left\|  V_m(t,\cdot)  \right\|^k_{\ell^k(K_m)}\right) + \E \left(\left\|  V_m(t,\cdot)  \right\|^k_{\ell^k(K\setminus K_m)}\right)\\
\le&  \left( \kappa (k-1+\Lambda)+\frac{k(k-1)}{2}\gamma(0)\Lip_\rho^2\right) \int_0^t  \E \left(\left\|  V_m(s,\cdot)  \right\|^k_{\ell^k(K)}\right) \ud s \\
& +\Lambda  \int_0^t \E \left( \sum_{j\in K\setminus K_m} U(s, j)^k\right) \ud s+\int_0^t 
\E \left(\Norm{V_m(s,\cdot)}_{\ell^k(K\setminus K_m)}^k \right)
\ud s.
\end{align*}
Thanks to \eqref{E:Uk} and \eqref{E_:VmKKm}, an application of Gronwall's lemma to 
\[
W_m(t):=\sup_{s\in[0,t]}\E\left(\Norm{V_m(t,\cdot)}_{\ell^k(K)}^k\right)
\]
proves the proposition. 
\end{proof}
 
%\begin{remark}
%Let $t>0$ and $i\in K$ be fixed. Then, Theorem \ref{T:Approx1} implies that $u^{(m)}(t,i)$ converges to $u(t,i)$ in probability as $m\to \infty$. Thus, \eqref{E:MomBDm} implies that for any $k\geq 2$
%\begin{equation}
% \lim_{m\to \infty} \E \left| u^{(m)}(t,i) - u(t,i)  \right|^k =0.
%\end{equation}
%\end{remark}

\subsection{Comparison theorems for finite interacting diffusions}\label{SS:Finite}
In this subsection, we will prove Theorems \ref{T:FiniteCc} and \ref{T:FiniteCc2}.  
Before the proof, we first make a remark to comment the difference of our results with those in  Cox, Fleischmann and Greven \cite{CFG96}. 

\begin{remark}\label{R:SDEs}
Here is a detailed comparison of our results --- both Theorem \ref{T:MomComSDE1} and Theorem \ref{T:FiniteCc} --- with Theorem 1 of Cox, Fleischmann and Greven \cite{CFG96}.
Let us  first mention that $\mathbf{F}$ and $\mathbf{F}_0$ in \cite{CFG96} correspond to 
$\mathbb{F}[C_b^{2,v}]$ and $\mathbb{F}[C_{b,\pm}^{2,v}]$, respectively, in our paper.
\begin{enumerate}[(a)]
\item For the infinite dimensional case, Theorem \ref{T:MomComSDE1} is able to cover a bigger function cone, namely, $\mathbb{F}[C_{p,+}^{2,v}]$ and $\mathbb{F}[C_p^{2,v}]$ in contrast with $\mathbb{F}[C_{b,+}^{2,v}]$ and $\mathbb{F}[C_b^{2,v}]$, respectively, in \cite{CFG96};
see \eqref{E:Rel-Fs} for relations of these function spaces. 
In particular, the multiple-time comparison result in Theorem \ref{T:MomComSDE1} works well for the moment functions $\mathbb{F}_M$. However, in order to apply the same comparison results in \cite{CFG96} to the moment functions, one needs to  restrict the moment functions to bounded subinterval $I \subset \R_+$, i.e., one needs to replace $\R_+^K$ in the definition \eqref{E:FM} by $I^K$; see Example 6 of \cite{CFG96}. We can make this extension thanks to our stronger approximation results in Section \ref{SS:Aprox2}, namely, Propositions \ref{P:rhoBD}, \ref{P:rho2} and \ref{P:Approx1}.
\item For the case of finite dimensional SDE with $C_c^2(\R_+)$ diffusion coefficient, Theorem \ref{T:FiniteCc} corresponds to Sections 2.1 -- 2.4 of \cite{CFG96}. 
By stating our results in terms of $\mathbb{F}[C^{2,v}]$, $\mathbb{F}[C_+^{2,v}]$ and $\mathbb{F}[C_-^{2,v}]$, our results are slightly more general, even thought this improvement is not essential because each component of the the diffusion process will live in the compact support of $\rho$. The major difference here (and also for infinite dimensional SDE case) is that in \cite{CFG96}, only the case of independent Brownian motions was studied. So we need to change the infinitesimal generator from 
 \begin{align}\label{E:GlInd}
G_\ell= \kappa \sum_{1\leq i, j \leq d} \left(p_{i,j}-\delta_{i,j}\right) x_j D_i +\frac{1}{2}\sum_{1\leq i \leq d} \rho_\ell^2(x_i) D_i^2,\quad \ell=1,2,
\end{align} 
to \eqref{E:Gl} below. This change won't bring any new difficulties. The original proof in \cite{CFG96} works line by line. 
\end{enumerate}
\end{remark}

Although Theorem \ref{T:FiniteCc} can be proved in the same way as those in Sections 2.1--2.4 of \cite{CFG96} with only minor changes as is explained in part (b) of the above remark, considering that Theorem \ref{T:FiniteCc2} is new, we will streamline the proof of both results altogether. This will also serve as an alternative presentation of the proofs in \cite{CFG96}.

\begin{proof}[Proof of Theorems \ref{T:FiniteCc} and \ref{T:FiniteCc2}]
Let the index set $K$ be $\{1,\cdots, d\}$.
Under both Assumptions \ref{A:SDE} and \ref{A:FiniteCc}, we have a finite dimensional SDE with $\rho \in C_c^2(\R_+)$. Hence, it is well-known that there exists a unique strong solution $U(t,\cdot) \in \R^d$.
For $\ell\in \{1,2\}$, let $U_\ell$ be the unique strong solution either corresponding to $\rho_\ell$ in case of Theorem \ref{T:MomComSDE1} or to $\gamma_\ell$ in case of Theorem \ref{T:MomComSDE2}.
In the following, we will slightly abuse the notation for the expectation. We may put subscript to denote the initial data and where there is no subscript, the initial data is $u_0(\cdot)$.

\bigskip

Now we need to prove the following two statements:
\begin{enumerate}
 \item For any integer $m\ge 1$,  $0<t_1<\cdots <t_m<\infty$, and 
 \[
 \text{either}\quad 
 F_1,\cdots F_m\in\mathbb{F}^{\{1,\cdots,d\}}[C_-^{2,v}]
 \quad\text{or}\quad 
 F_1,\cdots F_m\in\mathbb{F}^{\{1,\cdots,d\}}[C_+^{2,v}],
 \]
 it holds that
\begin{align}\label{E:MomComfin1}
 \E\left[\prod_{\ell=1}^m F_\ell\left( U_1 (t_\ell, \cdot)\right)\right]
 \ge 
 \E\left[\prod_{\ell=1}^m  F_\ell\left( U_2 (t_\ell, \cdot)\right)  \right].
\end{align}  
\item If $F$ is only in $\mathbb{F}^{\{1,\cdots,d\}}[C^{2,v}]$, then for any $t\geq 0$, 
\begin{align}\label{E:MomComfin11}
 \E\left[ F(U_1(t, \cdot))\right] \geq  \E\left[F(U_2(t, \cdot))\right].
 \end{align}
\end{enumerate}

{\noindent\bf Step 1.~}
We start by proving \eqref{E:MomComfin11} for  $F\in\mathbb{F}^{\{1,\cdots,d\}}[C^{2,v}]$, 
which will cover both \eqref{E:MomComfin11} and 
\eqref{E:MomComfin1} when all $t_\ell$ are the same.

For $\ell\in \{1,2\}$, let $G_\ell$ be the infinitesimal generator for $u_\ell(t,\cdot) \in \R^d$, that is, 
\begin{align}\label{E:Gl}
G_\ell= \kappa &\sum_{1\leq i, j \leq d} \left(p_{i,j}-\delta_{i,j}\right) x_j D_i+
\begin{cases}
\displaystyle
\frac{1}{2}\sum_{1\leq i, j \leq d} \rho_\ell(x_i)\rho_\ell(x_j) \gamma(i-j) D_iD_j,& \text{Theorem \ref{T:FiniteCc}},\\[2em]
\displaystyle
\frac{1}{2}\sum_{1\leq i, j \leq d} \rho(x_i)\rho(x_j) \gamma_\ell(i-j) D_iD_j,& \text{Theorem \ref{T:FiniteCc2}}.
\end{cases}
\end{align} 
Let $T_t^{(\ell)}$ be the corresponding semigroup, namely, 
\begin{align}\label{E:TF=EFU}
T_t^{(\ell)} F(x):=\E_{x}\left[ F\left(U_\ell(t,\cdot)\right)\right]\quad \text{for all $x\in \R^d$.}
\end{align}
Then, \eqref{E:MomComfin11} is equivalent to showing 
\begin{equation}\label{E:ComSG}
T_t^{(1)} F(x) \geq T_t^{(2)} F(x) \quad \text{for all $x\in \R^d$}.
\end{equation}
By the integration by parts formula, 
\[ T_t^{(1)} - T_t^{(2)} = \int_0^t T_s^{(1)}\left[ G^{(1)} - G^{(2)} \right] T^{(2)}_{t-s} \,\, \ud s,\]
where 
\begin{equation*}
  G^{(1)} - G^{(2)}=
 \begin{cases}
 \displaystyle
 \frac{1}{2}\sum_{1\leq i, j \leq d} \left[\rho_1(x_i)\rho_1(x_j)-\rho_2(x_i)\rho_2(x_j)\right] \gamma(i-j) D_iD_j,
 & \text{Theorem \ref{T:FiniteCc}},\\[2em] 
 \displaystyle
 \frac{1}{2}\sum_{1\leq i, j \leq d} \rho(x_i)\rho(x_j)\left[\gamma_1(i-j) - \gamma_2(i-j) \right]   D_iD_j,
 & \text{Theorem \ref{T:FiniteCc2}}. 
 \end{cases}
\end{equation*}
It is clear that $T^{(\ell)}_t$ preserves positivity, i.e., $T^{(\ell)}_t g \geq 0$ whenever $g\geq 0$.
It is also known (see, e.g., Theorem 5.6.1 in \cite{F75}   that under our assumption on $\rho_\ell$ or $\rho$,  
\begin{align}\label{E:PreserveC2}
T^{(\ell)}_t F \in C^{2}(\R^d) 
\end{align}
and $T^{(\ell)}_t F$ is continuous in $t$.
Our assumptions on $\rho$'s and $\gamma$'s assure that for all $i,j\in\{1,\cdots,d\}$ and all $x_i,x_j\in\R_+$, 
\[
\begin{cases}
\displaystyle
\rho_1(x_i)\rho_1(x_j)-\rho_2(x_i)\rho_2(x_j)\ge 0,  &  \text{Theorem \ref{T:FiniteCc}},\\[1em]
\rho(x_i)\rho(x_j)\left[\gamma_1(i-j) - \gamma_2(i-j)\right]\ge 0, &\text{Theorem \ref{T:FiniteCc2}}.
\end{cases}
\]
Hence, we only need to show that $D_iD_j T_t^{(2)} F (z)\geq 0$ for all $1\leq i, j\leq d$, $t> 0$ and $z\in\R_+^d$. For simplicity, we define $U(t, \cdot):=U^{(2)}(t, \cdot)$, $\rho:=\rho_2$, $G:=G^{(2)}$ and $T_t:=T^{(2)}_t$,  and show that  
\begin{equation}\label{E:DT}
D_iD_j T_t F(z) \geq 0 \quad \text{for all $1\leq i, j\leq d$,  $t> 0$ and $z\in\R_+^d$}.
\end{equation} 

{\bigskip\bf\noindent Step 2.~}
In this step, we will use {\it Trotter's product formula} (see, e.g., Corollary 1.6.7 of \cite{EK86}) to prove \eqref{E:DT}. 
Let $T^{(\kappa,\rho)}_t$ denote this semigroup of the $d$-dimensional diffusion process in \eqref{E:SDE1} with drift parameter $\kappa$ and diffusion coefficient $\rho$.
Trotter's product formula suggests to study the limit of the semigroup $\left[ T_{t/k}^{(\kappa,0)} T_{t/k}^{(0,\rho)}\right]^k$ as $k\rightarrow\infty$.

We first study the semigroup $T^{(0,\rho)}_t$, i.e., the case where $\kappa=0$. In this case, \eqref{E:SDE1} becomes 
\begin{equation} \label{E:diffSDE}
\begin{cases}
 dU(t, i)=\rho\left(U(t, i) \right) d M_i(t), & t>0,\: i=1,\cdots, d,\\
 U(0,i) = u_0(i), & i=1,\cdots, d.
\end{cases}
\end{equation}
Although $U(t, i)$ and $U(t, j)$ are not independent, they interact only through the random environment $M_i(t)$ when $i\neq j$.  Hence, in \eqref{E:diffSDE} each component $U(t, i)$ of $(U(t, 1), \dots, U(t, d))$ has its own equation.
Following Cox {\it et al} \cite{CFG96}, for $i,j\in \{1,\cdots,d\}$, and $h_1, h_2>0$, denote 
\[
\begin{array}{lll}
    u^2 =z+h_2e_j && u^{12} = z+ h_1e_i +h_2e_j  \\
    u^0 =z && u^{1} = z+ h_1e_i,
\end{array}
\]
where $e_i$ is the $i$th unit vector in $\R^d$  and $z\in\R^d_+$.
To avoid triviality, we assume that $z\in \sprt(\rho)^d$.
When $i\ne j$, $(u^0,u^1,u^2,u^{12})$ forms a rectangle in the $(i,j)$-th directions; when $i=j$, it forms nondecreasing sequence in the $i$-th direction: $u^0\le u^1\wedge u^2\le u^1\vee u^2 \le u^{12}$ (here, the inequality  $u\geq v$ for $u, v \in \R^d$ means that each component $u_i \leq v_i$ for all $1\leq i\leq d$).
Let $U^0, U^1, U^2, U^{12}$ be the solutions to \eqref{E:diffSDE} with the initial condition $u^0, u^1, u^2, u^{12}$, respectively, when $i\ne j$ and with $u^0, u^1\wedge u^2, u^1\vee u^2, u^{12}$, respectively, when $i=j$.
By the classical comparison principle for the one-dimensional SDEs (see e.g., either \cite[Theorem VI.1.1]{IkedaWatanabe89} or \cite[Theorem IX.3.7]{RevuzYor99}), we have that with probability one, for all $t\ge 0$, 
\begin{align}\tag{in case of $i\ne j$}
&\begin{array}{ccc}
 U^2(t) &\le &U^{12}(t)\\
\mathbin{\rotatebox[origin=c]{90}{$\le$}}&& \mathbin{\rotatebox[origin=c]{90}{$\le$}}\\
U^0(t) &\le & U^1(t) 
\end{array}
\\[0.5em]
\tag{in case of $i=j$}
U^0(t) & \le U^1(t)  \le U^2(t)\le U^{12}(t).
\end{align}

Now for $F\in\mathbb{F}[C^{2,v}]$ we have that $D_iD_j F\ge 0$ for all $1\le i,j\le d$. By noticing that
\begin{align}\label{E:CalcDD}
D_iD_j f\ge 0 \quad \Longleftrightarrow \quad
f(u^{12}) - f(u^2)-f(u^1)+f(u^0)\ge 0 \quad \forall h_1, h_2>0,
\end{align}
we see that 
\begin{align}\label{E_:FFFF>0}
\left[ F\left(U^{12}(t)\right) - F\left((U^2(t)\right)\right] -\left[ F\left(U^{1}(t)\right) - F\left((U^0(t)\right)\right] \geq 0.
\end{align}
Notice that the expectation of the left-hand side of \eqref{E_:FFFF>0} is finite because $\rho$ has compact support and
\[
\max_{i=1,\cdots,d} U(t,i) \in \sprt(\rho),\quad a.s.
\]
Hence, we can take expectation on both sides of \eqref{E_:FFFF>0} to see that
\[
\left[T^{(0,\rho)}_t F(u^{12}) - T^{(0,\rho)}_t F(u^{1})
\right]
- \left[ T^{(0,\rho)}_t F(u^{2}) - T^{(0,\rho)}_t F(u^{0})\right] \ge 0
\]
which, in view of \eqref{E:CalcDD}, is nothing but \eqref{E:DT} for $T^{(0,\rho)}_t$. Therefore, we have proved \eqref{E:DT} for the case of $F\in \mathbb{F}[C^{2,v}]$ and no drift ($\kappa=0$).
In other words, $T^{(0,\rho)}_t$ preserves the function cone $\mathbb{F}[C^{2,v}]$.

Next, we study the semigroup $T^{(\kappa,0)}_t$, i.e., the case when $\kappa> 0$ but $\rho\equiv 0$ in \eqref{E:SDE1}.
In this case, the system is deterministic:
\begin{equation} 
\begin{cases}
 dU(t, i)=\kappa \sum_{j=1}^d (p_{i,j}-\delta_{i,j}) U(t,j) \ud t, & t>0,\: i=1,\cdots, d,\\
 U(0,i) = u_0(i), & i=1,\cdots, d.
\end{cases}
\end{equation}
If we view $U(t,\cdot)$ and the initial data $u_0(\cdot)$ as column vectors in $\R^d$ and set $A= (p_{i,j}-\delta_{i,j})_{1\le i,j\le d}$, then we have that $U(t,\cdot) = \exp(\kappa A t)  u_0(\cdot)$.
Hence, for $F\in\mathbb{F}[C^{2,v}]$ and $z\in\R_+^d$ (viewed as a column vector), 
\[
T^{(\kappa,0)}_tF(z) 
= F\left(\exp(\kappa A t) z \right)
\]
and hence,
\[
D_iD_j T^{(\kappa,0)}_t F(z)
=   \sum_{1\leq k, m \leq d}   F_{k,m}\left(\exp(\kappa A t) z \right) 
\left(\exp(\kappa A t)\right)_{k,i}\left(\exp(\kappa A t)\right)_{m,j}\ge 0,
\]
with $F_{k,m}(z) = D_kD_m F(z)$, which proves \eqref{E:DT} for this case. Therefore, $T^{(\kappa,0)}_t$ also preserves the function cone $\mathbb{F}[C^{2,v}]$.

Now we can apply Trotter's product formula with $C^2(\R_+^d)$ as the core to see that 
\[
\lim_{k\rightarrow\infty}
\left[ T_{t/k}^{(\kappa,0)} T_{t/k}^{(0,\rho)}\right]^k F
= T^{(\kappa,\rho)}_t F ,\qquad \forall F\in\mathbb{F}[C^{2,v}].
\]
By \eqref{E:PreserveC2}, the $C^2$-property is preserved by this semigroup $T^{(\kappa,\rho)}_t$. The nonnegativity of \eqref{E:DT} is also preserved through the limit. 
Therefore,  $T^{(\kappa,\rho)}_t$ preserves the function cone $\mathbb{F}[C^{2,v}]$.
This proves \eqref{E:MomComfin11} for $F\in\mathbb{F}[C^{2,v}]$
and \eqref{E:MomComfin1} when all $t_\ell$ are the same.

{\bf\bigskip\noindent Step 3.~~}
Notice that functions in $\mathbb{F}[C^{2,v}]$ is not closed under multiplication; see Example \ref{Ex:NotMult} below for one example. 
In order to work with multiple-time comparison which requires multiplication of these functions, we have to restrict to a smaller cone of convex functions. Since we already show in the previous step that  both function cones $\mathbb{F}[C^{2,v}_+]$ and $\mathbb{F}[C^{2,v}_-]$ are closed under the semigroup $T^{(\kappa,\rho)}_t$, we need only to show the preservation under multiplication and it is easy to see that both $\mathbb{F}[C^{2,v}_+]$ and $\mathbb{F}[C^{2,v}_-]$ are closed under multiplication. Indeed, if $F, G \in \mathbb{F}[C^{2,v}_+]$, then
\begin{align}\label{E_:DD>0}
D_iD_j (F G) = 
G (D_iD_j F)  + F (D_iD_j G) + (D_i F) (D_j G) + (D_j F) (D_i G) \ge 0 
\end{align}
because all terms $F$, $G$, $D_i F$, $D_j F$, $D_i G$, $D_j G$, $D_iD_j F$ and $D_iD_j G$ are nonnegative, the monotonicity is clearly preserved under product. Hence, $FG\in \mathbb{F}[C^{2,v}_+]$. The case for $\mathbb{F}[C^{2,v}_-]$ can be proved in the same way.

As a consequence, we claim that 
for any $F_1,\cdots,F_m\in\mathbb{F}[C^{2,v}_+]$ (resp. $\mathbb{F}[C^{2,v}_-]$) and $0\le t_1<\cdots <t_m$,  the function 
\begin{align}\label{E_:Step3}
z\mapsto \E_z\left[F_1(U(t_1))\cdots F_m(U(t_m)) \right]
\end{align}
belongs to $\mathbb{F}[C^{2,v}_+]$ (resp. $\mathbb{F}[C^{2,v}_-]$).
Indeed, the case $m=1$ has been proved in the previous step. Assume that this is true for $m-1$. Now by the strong Markov property, we see that 
\begin{align*}
\E_z\left[F_1(U(t_1))\cdots F_m(U(t_m)) \right]
&=
\E_z\left[F_1(U(t_1))\E_{U(t_1)}\left[F_2(U_{t_2-t_1})\cdots F_m(U(t_m-t_1)) \right]\right]\\
&= T^{(\kappa,\rho)}_{t_1}\left[
F_1 \: G \right](z)
\end{align*}
where $G(z) = \E_z\left[F_2(U_{t_2-t_1})\cdots F_m(U(t_m-t_1)) \right]$. By induction assumption, $G\in \mathbb{F}[C^{2,v}_+]$ (resp. $\mathbb{F}[C^{2,v}_-]$) Since $\mathbb{F}[C^{2,v}_+]$ (resp. $\mathbb{F}[C^{2,v}_-]$) is closed under multiplication, $F_1 G\in \mathbb{F}[C^{2,v}_+]$ (resp. $\mathbb{F}[C^{2,v}_-]$). This proves the claim in \eqref{E_:Step3}.

{\bf\bigskip\noindent Step 4.~~}
Now we will prove \eqref{E:MomComfin1} with $m\ge 2$. We need only to show one case, say $\mathbb{F}[C_{+}^{2,v}]$. 
The case $m=1$ has been proved in Step 2. Suppose that \eqref{E:MomComfin1} is true for $m-1$. For $m$, by the strong Markov property, for $\ell\in \{1,2\}$, we see that
\[
\E\left[F_1(U_\ell(t_1,\cdot) \cdots F_m(U_\ell(t_m,\cdot)) )\right]
= \E\left[F_1(U_\ell(t_1,\cdot) G_\ell(U_\ell(t_1,\cdot))\right]
\]
where 
\[
G_\ell (z) = \E_z\left[F_2(U_\ell(t_2-t_1,\cdot)\cdots F_m(U_\ell(t_m-t_1,\cdot))) \right].
\]
By the induction assumption, $G_1 (z)\ge G_2(z)$ for any $z\in\R_+^d$. Hence,
\begin{align*}
\E\left[F_1(U_1(t_1,\cdot) \cdots F_m(U_1(t_m,\cdot)) )\right]
=&  \E\left[F_1(U_1(t_1,\cdot) G_1(U_1(t_1,\cdot))\right] \\
\ge &  \E\left[F_1(U_1(t_1,\cdot) G_2(U_1(t_1,\cdot))\right] \\
\ge & \E\left[F_1(U_2(t_1,\cdot) G_2(U_2(t_1,\cdot))\right]\\
=&
\E\left[F_1(U_2(t_1,\cdot) \cdots F_m(U_2(t_m,\cdot)) )\right]
\end{align*}
Here we have used the fact that $F_1 G_2\in\mathbb{F}[C_{+}^{2,v}]$ (see Step 3). This proves \eqref{E:MomComfin1}.
\end{proof}

\subsection{Proof of Theorems \ref{T:MomComSDE1},  \ref{T:MomComSDE2} and Corollary \ref{C:Slepian}}
Now we are ready to prove Theorems \ref{T:MomComSDE1} and \ref{T:MomComSDE2}. \label{SS:4.4}

\begin{proof}[Proof of Theorems \ref{T:MomComSDE1} and \ref{T:MomComSDE2}]
Propositions \ref{P:rhoBD}, \ref{P:rho2} and \ref{P:Approx1} imply that  the infinite dimensional system of diffusions \eqref{E:SDE1} can be approximated by finite dimensional systems of diffusions with $C_c^2(\R_+)$-diffusion coefficient.
Then by passing to the limit in the above approximations, 
the stochastic comparison statements can be extended to those in Theorems \ref{T:MomComSDE1} and \ref{T:MomComSDE2}. 
When we pass to the limit, some care is needed. 
Indeed, by Propositions \ref{P:rhoBD}, \ref{P:rho2} and \ref{P:Approx1}, we can find $U_\epsilon(t,i)$ such that
\begin{align}\label{E:UeU}
\lim_{\epsilon\downarrow 0_+} 
\sup_{t\in[0,T]} \E\left[\Norm{U_\epsilon(t,\cdot)-U(t,\cdot)}_{\ell^p(K)}^p\right] = 0,\quad\forall p\in\bbN,\: T>0,
\end{align}
where $U(t,\cdot)$ is the unique solution to \eqref{E:SDE1} and $U_\epsilon (t,\cdot)$ with $\epsilon$ fixed solves a finite-dimensional SDE with $C_c^2(\R_+)$-diffusion coefficient.

{\bigskip\bf\noindent Case I.~} We first consider the one-time comparison results over $F\in\mathbb{F}[C_p^{2,v}]$.
Because $F\in\mathbb{F}[C_p^{2,v}]$, one can find $m\in\bbN\setminus\{0\}$ and distinct $i_1,\cdots,i_m\in K$ such that,
by the mean-value theorem, we have that
\[
\left|F(U(t,\cdot)) - F(U_\epsilon(t,\cdot))\right|
\le \left|\bigtriangledown F \left(\xi\right)\right| \sum_{\ell=1}^m \left|U(t,i_\ell)-U_\epsilon(t,i_\ell)\right|,\quad 
\xi:=(1-c)U(t,\cdot) + cU_\epsilon(t,\cdot),
\]
where $c\in [0,1]$. 
For any $\beta\ge 1$, by Cauchy-Schwartz inequality, we see that
\begin{align*}
\Norm{F(U(t,\cdot)) - F(U_\epsilon(t,\cdot))}_\beta
\le & \big\|\: 
|\bigtriangledown F(\xi)|\: \big\|_{2\beta} \Norm{\sum_{\ell=1}^m \left|U(t,i_\ell)-U_\epsilon(t,i_\ell)\right| }_{2\beta}.
\end{align*}
By the growth condition \eqref{E:GrowthCp}, there are some constants $C>0$ and $k\in\bbN$ such that
\begin{align*}
 \left|\bigtriangledown F \left(\xi\right)\right|^{2\beta} \le &
%  C_\beta \left(1+ \sum_{\ell=1}^m \left|(1-c)U(t,i_\ell) + cU_\epsilon(t,i_\ell) \right|^{2k\beta}\right)\\
%  \le &
 C_\beta\left(1+ \Norm{(1-c)U(t,\cdot) + cU_\epsilon(t,\cdot)}_{\ell^{2k\beta}(K)}^{2k\beta}\right)\\
 \le& C_{\beta,k} \left(1+\Norm{U(t,\cdot)}_{\ell^{2k\beta}(K)}^{2k\beta}+ \Norm{U_\epsilon(t,\cdot)}_{\ell^{2k\beta}(K)}^{2k\beta}\right).
\end{align*}
Hence,
\[
\big\|\: 
|\bigtriangledown F(\xi)|\: \big\|_{2\beta}
\le C \left(1+ \E\left[\Norm{U(t,\cdot)}_{\ell^{2k\beta}(K)}^{2k\beta}\right]+ \E\left[\Norm{U_\epsilon(t,\cdot)}_{\ell^{2k\beta}(K)}^{2k\beta}\right] \right)^{1/(2\beta)}
\]
Thanks to \eqref{E:UeU}, 
\[
\sup_{ \epsilon\in(0,1)} \big\|\: 
|\bigtriangledown F(\xi)|\: \big\|_{2\beta}<+\infty.
\]
On the other hand, 
\[
\Norm{\sum_{\ell=1}^m\left|U(t,i_\ell)-U_\epsilon(t,i_\ell)\right| }_{2\beta} 
\le C_\beta\E\left[\Norm{U(t,\cdot)-U_\epsilon(t,\cdot)}_{\ell^{2\beta}(K)}^{2\beta}\right]^{1/(2\beta)}
\rightarrow 0,\quad \text{as $\epsilon\downarrow 0_+.$}
\]
Hence, $F(U_\epsilon(t,\cdot))$ converges to $F(U(t,\cdot))$ in $L^{\beta}(\Omega)$ for all $\beta\ge 1$. The comparison results will be carried thought the limit.

{\bigskip\bf\noindent Case II.~} Now we consider the $m$-time comparison results with $m\ge 2$. By telescoping and the H\"older's inequality, for any $\beta\ge 1$,
\begin{align*}
&\Norm{
\prod_{\ell=1}^m F_\ell(U(t_\ell,\cdot))
-
\prod_{\ell=1}^m F_\ell(U_\epsilon(t_\ell,\cdot))
}_\beta\\
&= \Norm{
\sum_{k=1}^m
\left[F_k(U_\epsilon(t_k,\cdot))-F_k(U_\epsilon(t_k,\cdot))\right]
\left(\prod_{\ell=1}^{k-1} F_\ell(U(t_\ell,\cdot))\right)
\left(\prod_{\ell=k+1}^m F_\ell(U_\epsilon(t_\ell,\cdot))\right)
}_\beta\\
& \le 
\sum_{k=1}^m \Norm{F_k(U_\epsilon(t_k,\cdot))-F_k(U_\epsilon(t_k,\cdot))}_{\beta m}
\left(\prod_{\ell=1}^{k-1} \Norm{F_\ell(U(t_\ell,\cdot))}_{\beta m}\right)
\left(\prod_{\ell=k+1}^m  \Norm{F_\ell(U_\epsilon(t_\ell,\cdot))}_{\beta m}\right)
\end{align*}
where we use the convention that product over an empty set gives one. 
If $F_1,\cdots,F_m\in\mathbb{F}[C_{b,-}^{2,v}]$, then
$F_\ell$ are bounded and globally Lipschitz continuous. Hence, the right-hand side of the above inequality goes to zero.  
On the other hand, if $F_1,\cdots,F_m\in\mathbb{F}[C_{p,+}^{2,v}]$, then by Case I we see that 
\[
\sup_{\epsilon
\in (0,1)}\Norm{F_\ell (U_\epsilon(t_\ell,\cdot))}_{\beta m}
\le
\sup_{\epsilon
\in (0,1)} \Norm{F_\ell (U_\epsilon(t_\ell,\cdot))-F_\ell (U(t_\ell,\cdot))}_{\beta m} + \Norm{F_\ell (U(t_\ell,\cdot))}_{\beta m}<\infty.
\]
Hence, 
\[
\Norm{
\prod_{\ell=1}^m F_\ell(U(t_\ell,\cdot))
-
\prod_{\ell=1}^m F_\ell(U_\epsilon(t_\ell,\cdot))
}_\beta\le C \sum_{k=1}^m \Norm{F_k(U_\epsilon(t_k,\cdot))-F_k(U_\epsilon(t_k,\cdot))}_{\beta m},
\]
which goes to zero as another application of Case I. Therefore, $\prod_{\ell=1}^m F_\ell(U_\epsilon(t_\ell,\cdot))$ converges to $\prod_{\ell=1}^m F_\ell(U(t_\ell,\cdot))$ in $L^{\beta}(\Omega)$ for all $\beta\ge 1$.
This completes the proof of both Theorem \ref{T:MomComSDE1} and \ref{T:MomComSDE2}.
\end{proof}

We now prove Corollary \ref{C:Slepian}.

\begin{proof}[Proof of Corollary \ref{C:Slepian}] 
We first consider the case under Assumption \ref{A:FiniteCc} (i.e.,   finite dimensional SDEs and $K:=\{1, \dots, d\}$).  Let 
\[ F(z_1, \dots, z_d) :=\prod_{k=1}^d \mathbf{1}_{(-\infty, a_k]}(z_k) \quad \text{and}\quad F_\epsilon (z_1, \dots, z_d):=\prod_{k=1}^d \phi_{\epsilon, k} (z_k),  \]
where $\phi_{\epsilon, k}(z_k) \in C^2(\R)$ are non-increasing and non-negative functions such that $\phi_{\epsilon, k}(z_k)$ converges to $\mathbf{1}_{(-\infty, a_k]}(z_k)$ as $\epsilon$ goes to 0 for each $z_k \in \R$. It is easy to see that $F_\epsilon$ is uniformly bounded by some constant, in $C^2(\R^d)$ and  $D_iD_j F_\epsilon \geq 0$ for $i\neq j$. On the other hand, the assumption that $\gamma_1(0)=\gamma_2(0)$ enables us to get 
\[ G^{(1)}-G^{(2)} = \frac{1}{2}\sum_{1\leq i\neq  j \leq d} \rho(x_i)\rho(x_j)\left[\gamma_1(i-j) - \gamma_2(i-j) \right]   D_iD_j,\]
where $G^{(i)}$ is the infinitesimal generator of $U^{(i)}$ as in the proof of Theorem \ref{T:FiniteCc}. Hence, following the proof of Theorem  \ref{T:FiniteCc}, we get 
\[ \E F_\epsilon\left(U_1(t, 1), \dots, U_1(t, d) \right) \geq  \E F_\epsilon\left(U_2(t, 1), \dots, U_2(t, d)\right).\] 
Therefore, thanks to the bounded convergence theorem, as $\epsilon$ goes to 0, we get 
\[  \E F (U_1(t, x_1), \dots, U_1(t, x_d) \geq  \E F (U_2(t, x_1), \dots, U_2(t, x_d),\]
which shows \eqref{E:Slepian1} under Assumption \ref{A:FiniteCc}. Under the  assumption in Theorem \ref{T:MomComSDE2} with $\gamma_1(0)=\gamma_2(0)$, the approximation results from Propositions \ref{P:rhoBD}, \ref{P:rho2} and \ref{P:Approx1} complete the proof of  \eqref{E:Slepian1}.

\end{proof}

\section{Some examples and one application}
\label{S:ExApp}

In all examples below, we always work either under the settings of Theorem \ref{T:MomComp} or under those of Theorem \ref{T2:MomComp}, and use $u_\ell(t,x)$, $\ell=1,2$, to denote corresponding solutions to \eqref{E:SHE}.

\begin{example}\label{Ex:NotMult}
For $n\in\bbN\setminus\{0\}$ and $c>0$, let $g_1(x)=(x-c)^{2n}$ and $g_2(x)=x^2$.
It is clear that $g_1 \in C_p^{2,v}(\R_+;\R_+)\setminus C_{p,\pm}^{2,v}(\R_+;\R_+)$ and $g_2\in C_{p,+}^{2,v}(\R_+;\R_+)$.
For any $t>0$ and $x_0\in\R^d$, denote $F_\ell(u(t,\cdot)) := g_\ell(u(t,x_0))$, $\ell=1,2$.
In this case, $F_1, F_2\in \mathbb{F}[C_p^{2,v}]$ but $F_1  F_2 \not\in \mathbb{F}[C_p^{2,v}]$ because 
\[
\frac{1}{2}\frac{\ud^2}{\ud x^2} (x-c)^{2n} x^2=
(x-c)^{2(n-1)}\left[( n+1) (2 n+1) x^2-2c
   (2 n+1) x+c^2\right]
\]
which is negative for some $x>0$ since the quadratic form has two positive solutions by noticing that $\Delta=4c^2n(1+2n)>0$. Nevertheless, since $F_1\in \mathbb{F}[C_p^{2,v}]$, we can make the one-time comparison statement 
\[
\E\left(\left[u_1(t,x_0)-c\right]^{2n}\right)
\ge 
\E\left(\left[u_2(t,x_0)-c\right]^{2n}\right).
\]
This proves \eqref{E:C5}.
\end{example}

\begin{example}(Examples in $\mathbb{F}[C_{b,-}^{2,v}]$)
Let $g_1(x)=\frac{1}{(1+x)^{c}}$ with $c\ge 1$, $g_2(x)=\log\frac{x+a}{x+b}$ with $a>b>0$, and $g_3(x)=e^{-\lambda x}$ with $\lambda>0$. 
It is easy to see that for $\ell=1,2,3$, 
\[
g_i(x)>0, \quad g_i'(x)<0,\quad\text{and}\quad g_i''(x)>0,\quad\forall x\in\R_+.
\]
Hence, we have 
\[
g_1,g_2,g_3 \in C_{b,-}^{2,v}(\R_+;\R_+).
\]
Therefore, one can apply either Theorem \ref{T:MomComp} or Theorem \ref{T2:MomComp} using these functions to obtain the multiple-time comparison result in \eqref{E:CompCoordinate}.
\end{example}

\begin{example}
For any $a,b,d\ge 1$ and $c\ge e$, denote $g_1(x):=x^b[\log(c+ x)]^a$ and $g_2(x)=x^d$. 
We claim that 
\[
g_1,g_2\in C_{p,+}^{2,v}(\R_+;\R_+).
\]
It is trivial for the $g_2$ case. As for the $g_1$ case, it is clear that $g_1(x)$ is nonnegative and strictly increasing for $x\ge 0$. We also claim that $g_1''(x)\ge 0$ for $x\ge 0$. Indeed, for any $x\ge 0$, 
\begin{align}\notag
g_1''(x)\ge 0 & \Longleftrightarrow
(a-1) a x^2+a x ((2 b-1) x+2bc) \log (c+x)+(b-1) b
   (c+x)^2 \log^2(c+x)
   \ge 0\\ \notag
   &\Longleftarrow
   a x ((2 b-1) x+2bc) +(b-1) b
   (c+x)^2 \ge 0 \\
   &\Longleftrightarrow 
   (a (2 b-1)+(b-1) b)x^2 +  2b c( a + b-1 )x+(b-1) b
   c^2 \ge 0,
   \label{E:g''}
\end{align}
where in the second step we have used the fact that $\log(c+x)\ge 1$ and $a,b\ge 1$.
Now we need  the following conditions:
\[
\begin{cases}
 (a (2 b-1)+(b-1) b) \ge 0\\
  b( a + b-1 ) \ge 0\\
  (b-1) b\ge 0
\end{cases}
\]
in order to make \eqref{E:g''} true for all $x\ge 0$. Clearly, these conditions are satisfied for $a,b\ge 1$.
On the other hand,
\[
0\le g_1'(x) = b x^{b-1} \log ^a(c+x)+\frac{a
   x^b \log ^{a-1}(c+x)}{c+x}\le (a+b)x^b.
\]
Hence, we have proved that $g_1\in C_{p,+}^{2,v}(\R_+;\R_+)$.
Therefore, we can apply Theorems \ref{T:MomComp} and \ref{T2:MomComp} to have multiple-time comparison statements using either $g_1$ or $g_2$ or both. This proves \eqref{E:MomComp} and also Case (E-4) in Section \ref{S:Intro}.
\end{example}

\begin{example}
For $x\in\R_+^m$, let $g(x) = |x|^{\alpha}$, where $|x|=\sqrt{x_1^2+\cdots+x_m^2}$ and $\alpha\ge2$. Because

\[
D_i g(x) = \alpha x_i |x|^{\alpha-2}\quad\text{and}
\quad |\bigtriangledown g(x)|= \alpha |x|^{\alpha-1}
\]
and 
\[
D_iD_j g(x) = 
\begin{cases}
\displaystyle
 \alpha |x|^{\alpha-4}\left(|x|^2 +(\alpha-2)x_i^2 \right) & i=j,\\
 \displaystyle
 \alpha(\alpha-2)x_ix_j |x|^{\alpha-4} & i\ne j,
\end{cases}
\]
we see that $g\in C_{p,+}^{2,v}(\R_+^m;\R_+)$.
Therefore, one can apply either Theorem \ref{T:MomComp} or Theorem \ref{T2:MomComp} using these functions to obtain the multiple-time comparison result in \eqref{E:C4}.
\end{example}
 
Finally, let us give one application of approximation results proved in this paper. Here we can give a straightforward proof of the weak sample path comparison principle, which was proved in \cite{CK14Comp,Mueller91} (for one dimensional case) and in \cite{CH18Comparison} for ($d$-dimensional case). 

\begin{theorem}[Weak sample path comparison principle]\label{T:WeakComp}
 Assume that $f$ satisfies Dalang's condition \eqref{E:Dalang} and the diffusion coefficient $\rho$ is globally Lipschitz continuous, which is not necessary to vanish at zero.
Let $u_1$ and $u_2$ be two solutions to \eqref{E:SHE}
with the initial measures $\mu_1$ and $\mu_2$ that satisfy \eqref{E:J0finite}, respectively. 
If $\mu_1 \le \mu_2$, then
\begin{equation}\label{E: W comp point}
\bbP \left(u_1(t,x)\leq u_2(t,x)\right)=1\,, \quad \text{for all}\  t \geq 0 \ \text{and}\  x\in \RR^d\,.
\end{equation} 
\end{theorem}
\begin{proof}[Sketch of the proof]
Set $v=u_1-u_2$. Then $v$ satisfies a SHE similar to \eqref{E:SHE} with $\tilde{\rho}$ that satisfies $\tilde{\rho}(0)=0$. It suffices to show that $v(t,x)\ge 0$ a.s. for all $(t,x)$ fixed. 
As is shown in Section \ref{SS:Step4}, one can find
$v_{\epsilon_1,\epsilon_1',\epsilon_2}^\delta(t,[x]_\delta)$ such that 
\[
\lim_{\epsilon_1'\rightarrow 0_+}\lim_{\epsilon_1\rightarrow 0_+}
\lim_{\epsilon_2\rightarrow 0_+} \lim_{\delta\rightarrow 0_+}
\Norm{v(t,x)- v_{\epsilon_1,\epsilon_1',\epsilon_2}^\delta(t,[x]_\delta)}_p = 0,\quad\forall p\ge 2,\: t>0,\: x\in\R^d.
\]
On the other hand, $v_{\epsilon_1,\epsilon_1',\epsilon_2}^\delta(t,[x]_\delta)$ solves the infinite-dimensional SDE \eqref{E:SDE1} with $\rho$ replaced by $\tilde{\rho}$ and with nonnegative and nonvanishing initial data.
By Theorem 1.1 of Gei\ss\  and Manthey \cite{GM94}, we know that $v_{\epsilon_1,\epsilon_1',\epsilon_2}^\delta(t,[x]_\delta)\ge 0$ a.s. Therefore, this nonnegativity property will be passed to $v$ through the limit. 
\end{proof}

\section*{Acknowledgements}
\addcontentsline{toc}{section}{Acknowledgements}
% The authors thank two anonymous referees for their careful reading of the paper and numerous valuable suggestions. 
The first author would thank Roger Tribe for suggesting this problem and sharing his unpublished note \cite{Tribe13} during L.C.'s thesis defense in 2013. It has taken us several years to prepare several tools to solve this problem in this general setting.


\begin{thebibliography}{999}

% \bibitem{Adam03SecondEd}
% Adams, Robert A. and John J. F. Fournier.
% \newblock {\em Sobolev spaces}~(2$^{\text{nd}}$ ed.).
% \newblock Elsevier/Academic Press, Amsterdam, 2003.
%
% \bibitem{BainovSimeonov92}
% D.~Ba{\u\i}nov and P.~Simeonov.
% \newblock {\em Integral inequalities and applications}, volume~57 of {\em
%   Mathematics and its Applications (East European Series)}.
% \newblock Kluwer Academic Publishers Group, Dordrecht, 1992.
% 

\bibitem{BC16}
Balan, Raluca M. and Le Chen
\newblock Parabolic Anderson model with space-time homogeneous Gaussian noise and rough initial condition
\newblock {\em  J. Theoret. Probab.}, 31 (2018), no. 4, 2216--2265.

\bibitem{BC14Moment}  
Borodin, Alexei and Ivan Corwin.
\newblock Moments and Lyapunov exponents for the parabolic Anderson model.
\newblock {\em Ann. Appl. Probab.,} 24 (2014), no. 3, 1172--1198.


\bibitem{CarmonaMolchanov94}
Carmona, Ren\'e~A.  and Stanislav~A. Molchanov.
\newblock Parabolic Anderson problem and intermittency.
\newblock {\em Mem. Amer. Math. Soc.}, 108 (1994), no. 518, viii+125 pp. 
%Mem. Amer. Math. Soc. 
%
% \bibitem{LeChen13Thesis}
% Chen, Le.
% \newblock {\em Moments, intermittency, and growth indices for nonlinear
%   stochastic PDE's with rough initial conditions}.
% \newblock PhD thesis, No. 5712, \'Ecole Polytechnique F\'ed\'erale de Lausanne,
% 2013.
% %
\bibitem{ChenDalang13Heat}
Chen, Le and Robert C. Dalang.
\newblock Moments and growth indices for nonlinear stochastic heat equation
  with rough initial conditions.
\newblock {\em Ann.\ Probab.}  43 (2015), no. 6, 3006--3051.

% \bibitem{ChenDalangHolder}
% Chen, Le and Robert C. Dalang.
% \newblock H\"older-continuity for the nonlinear stochastic heat equation with rough initial conditions. 
% \newblock{\it Stoch. Partial Differ. Equ. Anal. Comput.} 2 (2014), no. 3, 316--352. 


\bibitem{CH18Comparison}
Chen, Le, Jingyu Huang.
\newblock Comparison principle for stochastic heat equation on $\R^d$.
\newblock {\em Ann.\ Probab.}, 47 (2019), no. 2, 989--1035.


% \bibitem{CHKK16Bad}
% Chen, Le, Jingyu Huang, Davar Khoshnevisan and Kunwoo Kim.
% \newblock Dense blowup for parabolic SPDEs.
% \newblock {\em In preparation, } 2016.


\bibitem{CK15SHE}
Chen, Le and Kunwoo Kim.
\newblock Nonlinear stochastic heat equation driven by spatially colored noise: moments and intermittency.
\newblock {\em  Acta Math. Sci. Ser. B}, 39B (2019), No. 3, 645--668.

\bibitem{CK14Comp}
Chen, Le and Kunwoo Kim.
\newblock On comparison principle and strict positivity of solutions to the nonlinear stochastic fractional heat equations.
\newblock {\em  Ann. Inst. Henri Poincaré Probab. Stat.}, 53 (2017), no. 1, 358--388.

% \bibitem{ChungWilliams90}
% Chung Kailai, and Ruth~J. Williams.
% \newblock {\em Introduction to stochastic integration}~(2$^{\text{nd}}$ ed.).
% \newblock Birkh\"auser Boston Inc., Boston, MA, 1990.

\bibitem{XChen15}
Chen, Xia.
\newblock Precise intermittency for the parabolic Anderson equation with an $(1+1)$-dimensional time-space white noise. 
\newblock {\it Ann. Inst. Henri Poincar\'e Probab. Stat.} 51 (2015), no. 4, 1486--1499. 

\bibitem{XChen19}
Chen, Xia.
\newblock Parabolic Anderson model with rough or critical Gaussian noise.
\newblock {\it Ann. Inst. Henri Poincar\'e Probab. Stat.} 55 (2019), no. 2, 941--976.

% \bibitem{CJK12}
% Conus, Daniel, Mathew Joseph and Davar Khoshnevisan.
% \newblock Correlation-length bounds, and estimates for intermittent islands in parabolic SPDEs. 
% \newblock {\em Electron. J. Probab.} 17 (2012), no. 102, 15 pp.

% \bibitem{ConusKhosh10Farthest}
% Conus, Daniel and Davar Khoshnevisan.
% \newblock On the existence and position of the farthest peaks of a family of
%   stochastic heat and wave equations.
% \newblock {\em Probab. Theory Related Fields }, {\bf 152}{\it (3-4)} (2012), 681--701.

\bibitem{CFG96}
Cox, J. Theodore, Klaus Fleischmann and Andreas Greven.
\newblock Comparison of interacting diffusions and an application to their 
ergodic theory. 
\newblock {\it Probab. Theory Related Fields} 105 (1996), no. 4, 513--528. 


\bibitem{Dalang99} Dalang, Robert C.
\newblock Extending the martingale measure stochastic integral with applications to spatially homogeneous s.p.d.e.'s.
\newblock {\it Electron.\ J. Probab.} 4 (1999), no. 6, 29 pp.
%
% \bibitem{DalangFrangos98}
% Dalang, Robert C. and Nicholas E. Frangos.
% \newblock The stochastic wave equation in two spatial dimensions.
% \newblock {\em Ann. Probab.}, {\bf 26}{\it (1)} (1998), 187--212.
%
% \bibitem{DalangMueller03} Dalang, Robert C. and Carl Mueller.
% 	\newblock Some non-linear S.P.D.E.'s that are second order in time.
% 	\newblock {\it Electron.\ J. Probab.}\ {\bf 8}{\it (1)} (2003) 21 pp. (electronic).
%
% \bibitem{DonohoStark} Donoho, David L. and Philip B. Stark.
% 	\newblock Uncertainty principles and signal recovery.
% 	\newblock {\it SIAM J. Appl.\ Math.}\ {\bf 49}{\it (3)} (1989) 906--931.
% %
% \bibitem{DalangQuer} 
% Dalang, Robert C. and Llu\'is Quer-Sardanyons.
% \newblock Stochastic integrals for spde's: a comparison.
% \newblock {\it Expo. Math.} 29 (2011), no. 1, 67--109. 
% 

% 
% \bibitem{DS80}
% Dawson, Donald A. and Habib Salehi.
% \newblock Spatially homogeneous random evolutions. 
% \newblock{\it J. Multivar. Anal.} 10 (1980), no. 2, 141--180. 

\bibitem{EK86}
Ethier, Stewart and Thomas G. Kurtz.
\newblock {\it Markov processes: Characterization and convergence.}
\newblock Wiley Series in Probability and Mathematical Statistics: Probability and Mathematical Statistics. John Wiley \& Sons, Inc., New York, 1986.

% \bibitem{Flores} Flores, G. R. Moreno.
% \newblock On the (strict) positivity of solutions of the stochastic heat equation. 
% \newblock {\it Ann. Probab.}  42 (2014), no. 4, 1635--1643.


\bibitem{FJL18} Foondun, Mohammud, Mathew Joseph and Shiu-Tang Li.
\newblock An approximation result for a class of stochastic heat equations.
\newblock {\it  Ann.\ Appl.\ Probab.}, to appear, 2018.



%
% \bibitem{FK14Cor} Foondun, Mohammud and Davar Khoshnevisan.
% \newblock Corrections and improvements to: ``On the stochastic heat equation with spatially-colored random forcing''.
% \newblock {\it Trans. Amer. Math. Soc.} 366 (2014), no. 1, 561--562.
%
\bibitem{FK13SHE} Foondun, Mohammud and Davar Khoshnevisan.
\newblock On the stochastic heat equation with spatially-colored random forcing.
\newblock {\it Trans. Amer. Math. Soc.} 365 (2013), no. 1, 409--458.
%
%  \bibitem{FLO14MB} Foondun, Mohammud and  Liu Wei and Omaba McSylvester.
% 	\newblock Moment bounds  for a class of fractional stochastic heat equations.
% 	\newblock Preprint available at \texttt{arXiv:1409.5687} (2014).
%
% \bibitem{FK} Foondun, Mohammud and Davar Khoshnevisan.
% 	\newblock On the global maximum of the solution to a stochastic
% 		heat equation with compact-support initial data.
% 	\newblock {\it Ann.\ Inst.\ Henri Poincar\'e Probab.\ Stat.}\
% 		{\bf 46}{\it (4)} (2010) 895--907.
%
% \bibitem{FK09EJP} Foondun, Mohammud and Davar Khoshnevisan.
% \newblock Intermittence and nonlinear parabolic stochastic partial differential equations.
% \newblock {\it Electron. J. Probab.} {\bf 14}{\it (21)} (2009), 548--568.

\bibitem{F75}
Friedman, Avner
\newblock {\it Stochastic differential equations and applications. Vol. 1. }
\newblock Probability and Mathematical Statistics, Vol. 28. Academic Press, New York-London, 1975. 

\bibitem{GM94}
Gei\ss, Christel and  Ralf Manthey.
\newblock Comparison theorems for stochastic differential equations in finite and infinite dimensions. 
\newblock {\it Stochastic Process. Appl.} 53 (1994), no. 1, 23--35. 


% \bibitem{GP15KPZ} 
% Gubinelli, Massimiliano and Nicolas Perkowski.
% \newblock KPZ reloaded.
% \newblock {\it Comm. Math. Phys.} 349 (2017), no. 1, 165--269.

		%
%\bibitem{Grafakos} Grafakos, Loukas. 
%\newblock {\it Classical Fourier analysis}. 
%\newblock Second edition. Graduate Texts in Mathematics, 249. Springer, New York, 2008. xvi+489 pp
%
% \bibitem{HHN}
% Hu, Yaozhong, Jingyu Huang and David Nualart. 
% \newblock On the intermittency front of stochastic heat equation driven by colored noises. 
% \newblock {\it Electron. Commun. Probab.} 21 (2016), no. 21, 13 pp.
% 
% 
% %
% \bibitem{HHNT} 
% Hu, Yaozhong, Jingyu Huang, David Nualart and Samy Tindel.
% \newblock Stochastic heat equations with general multiplicative Gaussian noises:
% H\"older continuity and intermittency.
% \newblock {\it Electron. J. Probab.} 20 (2015), no. 55, 50 pp.
% %

\bibitem{Huang}
Huang, Jingyu.
\newblock On stochastic heat equation with measure initial data.
\newblock {\em Electron. Commun. Probab.}, 22 (2017), no. 40, 6 pp.

% \bibitem{HLN}
% Huang, Jingyu, Khoa L\^e and David Nualart. 
% \newblock Large time asymptotics for the parabolic Anderson model driven by spatially correlated noise.
% \newblock {\em  Ann. Inst. Henri Poincaré Probab. Stat.} to appear, 2016. 


\bibitem{IkedaWatanabe89}
Ikeda, Nobuyuki and Shinzo Watanabe.
\newblock {\it Stochastic differential equations and diffusion processes.} Second edition. 
\newblock North-Holland Publishing Co., Amsterdam; Kodansha, Ltd., Tokyo, 1989.


\bibitem{Isserlis}
Isserlis, Leon.
\newblock On a formula for the product-moment coefficient of any order of a normal frequency distribution in any number of variables.
\newblock {\em Biometrika} 12 (1918), 1-2, 134--139.



\bibitem{JKM17}
Joseph, Mathew, Davar Khoshnevisan and Carl Mueller.
\newblock Strong invariance and noise-comparison principles for some parabolic stochastic PDEs
\newblock {\em  Ann.\ Probab.}, 45 (2017), no. 1, 377--403. 


\bibitem{KKX1} Khoshnevisan, Davar, Kunwoo Kim, and Yimin Xiao.
	\newblock Intermittency and multifractality: a case study via parabolic stochastic PDEs,
	\newblock {\it Ann.\ Probab.}\
	\newblock 45 (2017), no. 6A, 3697--3751. 

\bibitem{KKX2} Khoshnevisan, Davar, Kunwoo Kim, and Yimin Xiao,
	\newblock A macroscopic multifractal analysis of parabolic stochastic PDEs,
	\newblock {\it Comm.\ Math.\ Phys.}\
	\newblock 360 (2018), no. 1, 307--346. 

\bibitem{Kim} Kim, Kunwoo.
\newblock On the large-scale structure of the tall peaks for stochastic heat equations with fractional Laplacian,
\newblock {\it Stochastic \ Process. \ Appl.} 
\newblock129 (2019), no. 6, 2207--2227.


% %
% \bibitem{KZ} Karczewska, Anna and Jerzy Zabczyk. 
% \newblock Stochastic PDEs with function-valued solutions. 
% \newblock In {\it Infinite dimensional stochastic analysis (Amsterdam 1999)}, 197--216,


%
% \bibitem{khosh1}  Khoshnevisan, Davar.
% \newblock Analysis of stochastic partial differential equations.
% \newblock {\it CBMS Regional Conference Series in Mathematics}, 119. 
% American Mathematical Society, Providence, RI, 2014.

% \bibitem{KK13LCA}  Khoshnevisan, Davar and Kunwoo Kim.
% \newblock Non-linear noise excitation of intermittent stochastic PDEs and the topology of LCA groups.
% \newblock To appear in {\it Ann. Probab.} Preprint available at \texttt{arXiv:1302.3266} (2013).

%
% \bibitem{Liggett} Liggett, T. M.
% 	\newblock {\it Interacting Particle Systems}.
% 	\newblock Springer-Verlag, New York, 1985.

%
% \bibitem{Mueller2} Mueller, Carl.
% 	\newblock {\it Some Tools and Results for Parabolic Stochastic Partial
% 		Differential Equations} (English summary).
% 	\newblock In: A Minicourse on Stochastic Partial Differential Equations, 111--144,
% 	\newblock Lecture Notes in Math.\ {\bf 1962} Springer, Berlin, 2009.
% %


%%

%\bibitem{LM06}
%Lions, Pierre-Louis and Marek Musiela.
%\newblock Convexity of solutions of parabolic equations. 
%\newblock {\it C. R. Math. Acad. Sci. Paris} 342 (2006), no. 12, 915--921. 

\bibitem{Mueller91} Mueller, Carl.
\newblock On the support of solutions to the heat equation with noise.
\newblock {\it Stochastics Stochastics Rep.} 37 (1991), no. 4, 225--245.

%%
% \bibitem{MuellerNualart}
% Mueller, Carl and David Nualart. 
% \newblock Regularity of the density for the stochastic heat equation. 
% \newblock {\it Electron. J. Probab.}, 13 (2008), no. 74, 2248--2258.



% %
% \bibitem{Nobel97} Noble, John M.
% \newblock Evolution equation with Gaussian potential.
% \newblock {\it Nonlinear Anal.}\ {\bf 28}{\it (1)} (1997) 103--135.

% \bibitem{oldham2008atlas}
% Oldham Keith and J.~Myland and J.~Spanier.
% \newblock {\em An atlas of functions}.
% \newblock Springer, New York, second edition, 2009.
%
% \bibitem{NIST2010}
% F.~W.~J. Olver, D.~W. Lozier, R.~F. Boisvert, and C.~W. Clark, editors.
% \newblock {\em N{IST} handbook of mathematical functions}.
% \newblock U.S. Department of Commerce National Institute of Standar\ud s  and
%   Technology, Washington, DC. Cambridge Univ. Press, Cambridge, 2010.
%
% \bibitem{SanzSoleSarra02}
% Sanz-Sol\'e, Marta and M\`onica Sarr\`a.
% \newblock H\"older continuity for the stochastic heat equation with spatially correlated noise.
% \newblock {\em Seminar on Stochastic Analysis, Random Fields  and Applications, III (Ascona, 1999)},
% 259--268, Progr. Probab., 52, Birkhäuser, Basel, 2002.

% \bibitem{Spitzer1981} Spitzer, Frank.
% 	\newblock Infinite systems with locally interacting components.
% 	\newblock {\it Ann.\ Probab.}\ {\bf 9} (1981) 349--364.
%

% \bibitem{Stein70Singular}
% Elias~M. Stein.
% \newblock {\em Singular integrals and differentiability properties of
%   functions}.
% \newblock Princeton University Press, Princeton, N.J., 1970.
%

\bibitem{Shiga}
Shiga, Tokuzo.
\newblock Two contrasting properties of solutions for one-dimensional stochastic partial differential equations. 
\newblock {\it Canad. J. Math.} 46 (1994), no. 2, 415--437. 

\bibitem{ShigaShimuzu}
Shiga, Tokuzo and  Akinobu Shimizu.
\newblock Infinite-dimensional stochastic differential equations and their applications. 
\newblock {\it J. Math. Kyoto Univ.} 20 (1980), no. 3, 395--416. 
%
% \bibitem{TessitoreZabczyk98}
% Tessitore, Gianmario and Jerzy Zabczyk.
% \newblock Strict positivity for stochastic heat equations.
% \newblock {\it Stochastic Process. Appl.} 77 (1998), no. 1, 83--98.

\bibitem{Tribe13}
Tribe, Roger and Jerzy Zabczyk.
\newblock Notes on stochastic comparisons theorems for SDEs and SPDEs.
\newblock {\it Unpublished note}, 2013.

\bibitem{RevuzYor99}
 Revuz, Daniel and Marc Yor.
 \newblock {\it Continuous martingales and Brownian motion.} Third edition. 
 \newblock Grundlehren der Mathematischen Wissenschaften, 293. Springer-Verlag, Berlin, 1999.
 
 
\bibitem{Walsh} Walsh, John B.
	\newblock {\it An Introduction to Stochastic Partial Differential Equations}.
	\newblock In: \`Ecole d'\`et\'e de probabilit\'es de
	Saint-Flour, XIV---1984, 265--439.
	Lecture Notes in Math.\ 1180, Springer, Berlin, 1986.
%
\end{thebibliography}
\end{document}